\documentclass{article}

\usepackage{geometry}
\usepackage{graphicx}
\usepackage{amsmath}
\usepackage{amssymb}
\usepackage{color}
\usepackage{hyperref}
\usepackage{epsfig}
\usepackage{tikz}
\usepackage{pgfplots}
\usepackage{caption}
\usepackage{float}
\usepackage{subcaption}
\usepackage{enumitem}

\usepackage{bbm}
\usepackage{amsfonts}
\usepackage{amsmath, amsthm, amsfonts, amssymb, color}
\usepackage{mathrsfs}
\usepackage{color}

\allowdisplaybreaks

\newtheorem{thm}{Theorem}[section]

\newtheorem{remark}[thm]{Remark}
\theoremstyle{definition}

\newcommand{\scr}[1]{\mathscr #1}
\definecolor{wco}{rgb}{0.5,0.2,0.3}

\numberwithin{equation}{section} \theoremstyle{remark}

\newcommand{\ua}{\uparrow}

\newcommand{\var}{\mathbb{V}}
\newcommand{\di}{\,\textup{d}}

\newcommand{\NN}{{\mathbb N}}

\newcommand{\RR}{{\mathbb R}}

\newcommand{\notiz}[1]{}

\newcommand{\todo}[1]{\vspace{5 mm}
\par \noindent
\marginpar{\textsc{Simon}} 
\framebox{\begin{minipage}[c]{0.95
\textwidth}\raggedright \texttt{#1} \end{minipage}}
\vspace{5 mm}\par}

\newcommand{\todoStef}[1]{\vspace{5 mm}
\par \noindent
\marginpar{\textsc{Stefan}} 
\framebox{\begin{minipage}[c]{0.95
\textwidth}\raggedright \texttt{#1} \end{minipage}}
\vspace{5 mm}\par}

\title{
A mean-field extension of the LIBOR market model
\thanks{\emph{Disclaimer.} The opinions expressed in this article are those of the authors and do not necessarily reflect the official position of the Austrian Financial Market Authority.}\\ 
}   
\author{Sascha Desmettre\footnote{Insitute of Financial Mathematics and Applied Number Theory, Johannes Kepler University Linz, Altenbergerstrasse~69, A-4040 Linz, E-mail:sascha.desmettre@jku.at}, Simon Hochgerner\footnote{\"Osterreichische Finanzmarktaufsicht, Otto-Wagner-Platz 5, A-1090 Wien, simon.hochgerner@fma.gv.at}, Sanela Omerovic\footnote{\"Osterreichische Finanzmarktaufsicht, Otto-Wagner-Platz 5, A-1090 Wien, omerovic@alumni.tugraz.at}, 
 Stefan Thonhauser\footnote{Insitute of Statistics, Graz University of Technology, Kopernikusgasse 24/III, A-8010 Graz, E-mail:stefan.thonhauser@math.tugraz.at} 
}

\date{\today}

\begin{document}
\allowdisplaybreaks
\def\R{\mathbb R}  \def\ff{\frac} \def\ss{\sqrt} \def\B{\mathbf
B}
\def\N{\mathbb N} \def\kk{\kappa} \def\m{{\bf m}}
\def\ee{\varepsilon}\def\ddd{D^*}
\def\dd{\delta} \def\DD{\Delta} \def\vv{\varepsilon} \def\rr{\rho}
\def\<{\langle} \def\>{\rangle} \def\GG{\Gamma} \def\gg{\gamma}
  \def\nn{\nabla} \def\pp{\partial} \def\E{\mathbb E}
\def\d{\,\text{\rm{d}}} 
\def\bb{\beta} \def\aa{\alpha} \def\D{\scr D}
  \def\si{\sigma} \def\ess{\text{\rm{ess}}}
\def\beg{\begin} \def\beq{\begin{equation}}  \def\F{\scr F}
\def\Ric{\text{\rm{Ric}}} \def\Hess{\text{\rm{Hess}}}
\def\e{\text{\rm{e}}} \def\ua{\underline a} \def\OO{\Omega}  \def\oo{\omega}
 \def\tt{\tilde} \def\Ric{\text{\rm{Ric}}}
\def\cut{\text{\rm{cut}}} \def\P{\mathbb P} \def\ifn{I_n(f^{\bigotimes n})}
\def\C{\scr C}   \def\G{\scr G}   \def\aaa{\mathbf{r}}     \def\r{r}
\def\gap{\text{\rm{gap}}} \def\prr{\pi_{{\bf m},\varrho}}  \def\r{\mathbf r}
\def\Z{\mathbb Z} \def\vrr{\varrho} \def\ll{\lambda}
\def\L{\scr L}\def\Tt{\tt} \def\TT{\tt}\def\II{\mathbb I}
\def\i{{\rm in}}\def\Sect{{\rm Sect}}  \def\H{\mathbb H}
\def\M{\scr M}\def\Q{\mathbb Q} \def\texto{\text{o}} \def\LL{\Lambda}
\def\Rank{{\rm Rank}} \def\B{\scr B} \def\i{{\rm i}} \def\HR{\hat{\R}^d}
\def\to{\rightarrow}\def\l{\ell}\def\iint{\int}
\def\EE{\scr E}\def\no{\nonumber}
\def\A{\scr A}\def\V{\mathbb V}\def\osc{{\rm osc}}
\def\BB{\scr B}\def\Ent{{\rm Ent}}
\def\U{\scr U}\def\8{\infty} \def\si{\sigma}\def\1{\lesssim}

\maketitle 
\begin{abstract}
We introduce a mean-field extension of the LIBOR market model (LMM) which preserves the basic features of the original model. Among others, these features are the martingale property, a directly implementable calibration and an economically reasonable parametrization of the classical LMM. At the same time, the mean-field LIBOR market model (MF-LMM) is designed to reduce the probability of exploding scenarios, arising in particular in the market-consistent valuation of long-term guarantees. To this end, we prove existence and uniqueness of the corresponding MF-LMM and investigate its practical aspects, including a Black '76-type formula. Moreover, we present an extensive numerical analysis of the MF-LMM. The corresponding Monte Carlo method is based on a suitable interacting particle  system which approximates the underlying mean-field equation. 
\end{abstract} 

\section{Introduction}
LIBOR market models (LMMs) are nowadays widely used by practitioners to valuate market instruments which depend on interest rate movements, such as caps or swaptions. 
These models are popular due to their relative ease of calibration coupled with the possibility to economically interpret the relevant parameters. LMMs have been developed (\cite{MR97,Brigo,Rebonato}) with a view towards valuating interest rate derivatives with a maturity/tenor structure at the the order of months or a few years. 

Recently, the valuation of long term guarantees has become increasingly important in the life insurance sector. The regulatory framework Solvency~II \cite{L1}, which was implemented by the European Union per January 1, 2016, requires European insurers to assign market consistent values to their liability portfolios. For the case of life insurance with profit participation this means that cash flows have to be projected along arbitrage free scenarios to yield a Monte Carlo method of calculating the associated expected value (\cite{GH21,HG19,VELP17}). 
A life insurance portfolio may have a time to run-off of several decades and it is not unusual to have a projection horizon of $60$ years, or more. Indeed, the EIOPA risk free curve is published with a length of $120$ years (\cite{EIOPA_curve}).

For the aforementioned reasons LMMs have become very popular also in the insurance sector. However, because of the projection horizon of a typical life insurance portfolio, the generated scenarios often suffer from blow-up. In this context, blow-up or explosion means that there is a significant number of scenarios (e.g., more than $1\%$) such that the forward rate (for any maturity and any point in time) exceeds a predefined threshold (e.g., $50\%$ of interest). 
The explosion problem of LMMs is also theoretically well-known (see e.g. \cite{SG}). 

This is a practical problem for two reasons: Firstly, extremely high interest rate scenarios are unlikely since central banks act to stabilize rates around a given ultimate forward rate target. Hence, a significant percentage of exploding rates hints at  unrealistic scenario evolution. Secondly, explosion of rates over a longer period of time leads to discount factors below machine accuracy, resulting in a vanishing cash-flow in the Monte Carlo routine. 

To mitigate the explosion problem, there are two popular and practical approaches (\cite[p.~38f]{AP20}): 
\begin{itemize}
    \item 
    \emph{Volatility freeze:}
    If a scenario breaches a predefined threshold (e.g., $50\%$) then the scenario evolves according to the prevailing term structure from that time onward. That is, the volatility is formally set to $0$ in this scenario. 
    \item
    \emph{Capping:}
    If a scenario breaches a predefined cap (e.g., $70\%$) then the rates are set equal to this cap from that time forward.
\end{itemize}
While these methods are clearly very effective in avoiding explosion, there are caveats: The cap has negative consequences for the martingale properties of the set of risk-free scenarios, and both reduce the scenario implied volatility; which violates market-consistency. The German Association of Actuaries (DAV) outlines that a capping of exploding interest rates should be avoided as it violates in general the risk neutral framework (\cite{DAV1}, \cite{DAV2}). 

The no-arbitrage property is important in practice also as a necessary condition for the applicability of the so-called leakage test (see \cite{HG19}). 

In this paper, we work towards an extension of LMMs such that all the salient features (ease of calibration, economic interpretation of parameters, martingale property) are preserved, but with the added benefit that the probability of blow-up is significantly reduced. 
Therefore, we introduce \emph{mean-field LIBOR market models} (MF-LMMs). The essential idea here is that the mean-field dynamics can depend on properties of the empirically observed distribution of the scenario set. Specifically, we use the observed second moment as a measure for explosion and design the evolution equation in such a way that the growth of the solution's second moment is dampened. 

We remark that our main motivation for this mean-field extension is the valuation of long term guarantees. To this end, we provide a numerical study which shows that the probability of explosion is considerably reduced in the MF-LMM framework. This is viewed as evidence that the mean-field approach has, when properly implemented, effects which are desirable from the economic perspective. 
Moreover, in order to have a tractable extension of the standard LMM technique, we consider a standard LMM calibration and an a posteriori mean-field augmentation of the dynamics. Thus the parametrization of the mean-field component is justified by the economic plausibility of the resulting scenarios while the non-mean-field parameters follow from a standard calibration routine. 

Concerning the economic motivation for MF-LMMs we stress that interest rate movements are stabilized by central bank policy. In short rate models this stylized fact is often incorporated by means of a mean-reversion assumption. This is an observation which is also reflected by the literature strand on \textit{backward-looking rates} that disentangles LIBOR rates in a compounded overnight risk-free rate plus a fixed spread (see e.g. \cite{VPSine}, \cite{LM19, LM19_}), based on the ISDA fallback protocol (see \cite{ISDA1,ISDA2}) for LIBOR benchmarks. Whereas this LIBOR spread approach will be pertinent in the long-term future of fixed-term reference rate modeling, our approach focuses on existing long-term contracts based on the LIBOR that are still present and popular among practitioners.

\subsection{Outline}\label{sec:orga}
This paper is organized as follows: In Section~\ref{sec:model}, the MF-LMM is introduced and its existence is proved by showing that the corresponding mean-field SDE\footnote{For the remainder of the paper and consistently with e.g.~\cite{BLPR}, we use the terms mean-field SDE and McKean-Vlasov SDE synonymously.} is well-posed and has a unique strong solution. Section~\ref{SEC:Cali} addresses practical aspects of the MF-LMM such as a Black '76-type formula for a given measure flow, calibration and change to the spot measure. Section~\ref{sec:numerics} contains the numerical study and results, which show how the model can be used to reduce explosion. The effects on cap and swaption prices are studied, and it is shown that these can be (approximately) preserved by a judicial choice of the mean-field component.
Section~\ref{Sec:WEll} deals with existence and uniqueness of the solution to the underlying mean-filed SDE.

\subsection{Notation}\label{sec:prel}
For the remainder of the manuscript, we resort to the following terms and definitions:

\begin{itemize}
\item $(\Omega,\mathcal{F},(\mathcal{F}_t)_{t \geq 0},\mathbb{P})$ denotes a filtered probability space satisfying the usual conditions, see e.g., \cite{SAB}.
\item Let $(\mathbb{R}^d,\left \langle \cdot,\cdot \right \rangle, |\cdot|)$ be $d$-dimensional ($d \geq 1$) Euclidean space. As a matrix-norm, we use the Hilbert-Schmidt norm denoted by $\| \cdot \|$.
\item 
Let $\mathscr{P}(\mathbb{R}^d)$ to denote the family of all probability measures on $\mathbb{R}^d$ and define the subset of probability measures with finite second moment by
\begin{equation*}
    \mathscr{P}_2(\mathbb{R}^d)
    := 
    \Big \{ \mu\in \mathscr{P}(\mathbb{R}^d) : 
    \int_{\mathbb{R}^d} |x|^2 \mu(\mathrm{d} x) < \infty 
    \Big \}.
\end{equation*} 
\end{itemize}

\section{The mean-field LIBOR market model}\label{sec:model}
We start with the recapitulation of standard LIBOR market models:
The fair price  at time $t$ of a zero-coupon bond paying $1$ unit of currency at expiry date $T\ge t$ will be denoted by $P(t,T)$. 
We fix a tenor structure $0 \leq t_0 < t_1 < \ldots < t_N$, and remark that $t_N$ may be large.  
The $i$-th forward LIBOR rate (i.e., the rate valid on $[t_{i-1},t_i]$) at time $t \le t_{i-1}$ is defined as
\begin{equation*}
    L_t^{i}
    := 
    \frac{1}{\delta_i} 
    \frac{P(t,t_{i-1}) - P(t,t_i)}{P(t,t_i)}\,.
\end{equation*} 
where we assume accrual periods $\delta_i = t_i - t_{i-1}$.

We follow the backward induction approach (as e.g. presented in \cite{MR97}) to define the LIBOR dynamics, see also \cite{MR09,Fil}.
We therefore employ the following standard postulates:  
\begin{enumerate}
    \item[(1)] 
    The initial term structure $P(0,t_i)$ is positive and non-increasing, thus $L^i_0\ge0$ for all $i=1,\dots,N$. 
    \item[(2)]
    For each index $i=1,\dots,N$, there is an $\R^d$-valued volatility $\sigma_i$, a  forward measure $\mathcal{Q}^i$ and a corresponding $d$-dimensional Brownian motion $W^i$ ($d\ge1$), such that the dynamics of $L^i$ under $\mathcal{Q}^i$ is given by 
    \begin{align}
    \label{e:LMM}
        \mathrm{d}L_t^{i} 
        = L_t^{i}  \sigma_i^{\top} \di W^{i}_t.
    \end{align}
    \item[(3)]
    The Radon-Nikodym derivatives of the forward measures are naturally given by
    \begin{equation}
        \label{e:cons}
        \frac{d\mathcal{Q}^{i-1}}{d\mathcal{Q}^i}
        = \mathcal{E}_{t_{i-1}}
        \left( \int_{0}^{\cdot}
        \frac{\delta_i L^i_s}{\delta_i L^i_s + 1} \sigma_i^{\top}(s) \, \di W^i_s
        \right)\,.
    \end{equation} 
\end{enumerate}
Note that the $\mathcal{Q}^i$, for $i<N$, are fixed by backward induction starting from the terminal measure $\mathcal{Q}^N$.  

In the classical LIBOR market model the volatility structures are deterministic functions of time, that is $\sigma_i = \sigma_i(t)\in\R^d$ for $t\in[0,t_{i-1}]$. 
We lift this construction to the mean-field setting. 
Therefore, we choose
\begin{equation}
    \label{e:LMM-mf}
    \sigma_i = \sigma_i(t,\mu_{t}^{i})
\end{equation}
where $\mu_{t}^{i} = \text{Law}(L_t^{i})$ is the law of $L_
t^i$ under $\mathcal{Q}^i$. 
Consequently, the process given by \eqref{e:LMM} needs to be formulated as a mean-field SDE, see e.g.\ \cite{CD}. In order to obtain a well-formulated and applicable model, we have to derive the following three results: 
\begin{enumerate}
    \item[(1)]
    The mean-field version of the SDE \eqref{e:LMM} is well-posed.
    \item[(2)]
    The mean-field version of the SDE \eqref{e:LMM} can be transformed (for all $i\le N$) to a mean-field SDE under the terminal measure $\mathcal{Q}^N$.
    \item[(3)]
    The model should imply a Black '76-type formula for cap prices. 
\end{enumerate}
The first two theoretical items are proved in Theorem~\ref{TH:TH1} and the third practical item is shown in Theorem~\ref{thm:cap}.
 

\subsection{Specific volatility structure}\label{sec:vol_str}

For answering (1) and (2), one clearly needs to give the volatility~\eqref{e:LMM-mf} some specific structure, to be in line with the theoretical results obtained in Section~\ref{Sec:WEll}. Exemplary choices may depend on the variance of the LIBOR rates, e.g. we consider volatility structures of the form
\begin{equation}
    \label{e:vol_gen}
    \sigma_i(t,\mu_t^i) = \lambda^i(t,\var_{\mathcal{Q}^i}[L_t^{i}])
\end{equation}
where $\var_{\mathcal{Q}^i}[L_t^{i}]$ denotes the variance of $L_t^i$ under $\mathcal{Q}^i$ and $\lambda^i: [0,t_{i-1}]\times\R^+\to\R^d$ is a deterministic function.

In order to incorporate a dampening effect, we will assume in the numerical study in Section~\ref{sec:numerics} that
\begin{equation}
\label{eq:diffusion}
    \sigma_i(t,\mu_{t}^{i}) 
    = 
    \sigma_i^{(1)}(t) 
    \exp\Big(
    -\max\left\{\var_{\mathcal{Q}^i}[L_t^{i}] - \tilde{\sigma},0\right\}
    \Big),
\end{equation}
where $\sigma_i^{(1)}$ is modeled as a bounded function $\sigma_i^{(1)}: [0,t_{i-1}] \to \RR^d$. As in the classical model, this deterministic component is calibrated to market data, such as caplet prices. $\tilde{\sigma}>0$ represents a threshold volatility, motivated from historical data. 
This means that the volatility is tamed if $\var_{\mathcal{Q}^i}[L_t^{i}]$ increases beyond the threshold. We emphasize that our general existence and uniqueness result, Theorem~\ref{TH:TH1} covers this chosen parametrization; compare in particular Remark~\ref{rem:max}.

\subsection{Existence and uniqueness of the MF-LMM}
The existence and uniqueness of the classical LIBOR market model  is based on a backward induction argument. 
Therefore, we derive the dynamics of the $i$-th LIBOR rate under the terminal forward measure $\mathcal{Q}^N$ in the form
\begin{align}
\label{e:LMM3}
    \mathrm{d}L_t^{i} 
    = 
    L_t^{i}\Big( 
        b_i(t,L_t,\tilde{\mu}_{t}^{i}) \,\mathrm{d}t 
        + \tilde{\sigma}_i(t,\tilde{\mu}_t^{i})\,\mathrm{d}W^{N}_t \Big)
\end{align}
for $0 \leq t < t_{i-1}$.  
In this formula, 
$\tilde{\mu}_t^{i}$ is the law under $\mathcal{Q}^N$ of a process that is specified below. The drift $b_i$ and the volatility structure, $\tilde{\sigma}$, are also explicitly derived below, see equation~\eqref{e:LIBORRates}.
The occurrence of $L_t$ in $b_i$ indicates that, under $\mathcal{Q}^N$, the dynamics of $L^{i}$ will depend also  on the other LIBOR rates $L^j$, for $j \ge i$. 
  
Applying Girsanov's theorem, together with \eqref{e:cons}, implies
\begin{equation}
\label{e:girs-i}
    \di W^{i}_t 
    = \di W^{N}_t - \sum_{k=i+1}^{N} \frac{\delta_k L_t^k}{1 +\delta_k L_t^k} \sigma_k(t,\mu_t^k) \di t \,
\end{equation}
where $\sigma_k$ is given by \eqref{e:LMM-mf}. 
However, here the $\mu_t^k$ are still the laws of $L_t^k$ under $\mathcal{Q}^k$ and not $\mathcal{Q}^N$. 

When switching from $\mathcal{Q}^i$ to $\mathcal{Q}^N$ we need to consider the distribution of $L^k$ for $i < k \leq N$ under $\mathcal{Q}^N$. Thus, in the specific situation of \eqref{e:vol_gen}, the variances have to be calculated with respect to $\mathcal{Q}^N$:
\begin{align}
\label{e:V-trnsf}
    \var_{\mathcal{Q}^i}[L_t^{i}]
    = 
   \mathbb  E_{\mathcal{Q}^i}\Big[\Big(
        L_t^i - \mathbb E_{\mathcal{Q}^i}[L_t^i]
        \Big)^2\Big]  
    = 
    \mathbb E_{\mathcal{Q}^N}\Big[\Big(
     L_t^i - \mathbb E_{\mathcal{Q}^N}[L_t^i Z_t^{i,N}] \Big)^2 Z_t^{i,N} \Big] 
    =:
    \Psi_t^{i,N}\,,
\end{align}
where  
\begin{equation}
    Z_t^{i,N}
    :=
    \frac{\mathrm{d}\mathcal{Q}^i}{\mathrm{d}\mathcal{Q}^N}\Big|_{\mathcal{F}_t}
    = 
    \mathbb E_{\mathcal{Q}^N}\Big[ 
        \frac{\mathrm{d}\mathcal{Q}^i}{\mathrm{d}\mathcal{Q}^N}\Big| \mathcal{F}_t
        \Big].
\end{equation}
To this end, we start with $i=N-1$ and obtain from \eqref{e:cons} the specific form 
\[
    Z_t^{N-1,N}
    =
    \frac{\mathrm{d}\mathcal{Q}^{N-1}}{\mathrm{d}\mathcal{Q}^N}\Big|_{\mathcal{F}_t}
    = \mathcal{E}_{t}
    \left( \int_0^{\cdot}
    \frac{\delta_N L^N_s}{\delta_N L^N_s + 1}\sigma^N(s,\mu_s^N)^{\top} \di W^N_s 
    \right) \,.
\]
To show the difference to the classical situation, we look at $i = N - 2$ and obtain
\[
 Z^{N-2,N}_t 
 = 
 \frac{\mathrm{d}\mathcal{Q}^{N-2}}{\mathrm{d}\mathcal{Q}^{N-1}}\frac{\mathrm{d}\mathcal{Q}^{N-1}}{\mathrm{d}\mathcal{Q}^{N}} \Big|_{\mathcal{F}_t} 
 = 
 Z^{N-1,N}_t \frac{\mathrm{d}\mathcal{Q}^{N-2}}{\mathrm{d}\mathcal{Q}^{N-1}} \Big|_{\mathcal{F}_t}  
 =: Z^{N-1,N}_t Z^{N-2,N-1}_t  \,.
\] 
Consequently, the dynamics of $Z^{N-2,N}_t$ are given by 
\begin{align*}
\mathrm{d} Z^{N-2,N}_t 
&= 
Z^{N-1,N}_t \mathrm{d}Z^{N-2,N-1}_t 
    + Z^{N-2,N-1}_t \mathrm{d}Z^{N-1,N}_t  
    + \mathrm{d}[Z^{N-1,N}_t,Z^{N-2,N-1}_t] \\
&= 
Z^{N-1,N}_t Z^{N-2,N-1}_t \frac{\delta_{N-1} L_t^{N-1}}{1 + \delta_{N-1} L_t^{N-1}} \sigma_{N-1}(t,\mu_t^{N-1})^{\top} \mathrm{d}W_t^{N-1}  \\
& \quad\quad 
    + Z^{N-2,N-1}_tZ^{N-1,N}_t \frac{\delta_{N} L_t^{N}}{1 + \delta_{N} L_t^{N}} \sigma_N(t,\mu_t^{N})^{\top} \mathrm{d}W_t^{N}  
    + \mathrm{d}[Z^{N-1,N}_t,Z^{N-2,N-1}_t].
\end{align*}
Thus, in order to obtain the distribution of $L^{N-2}$ under $\mathcal{Q}^N$, we also need to consider the distribution of $Z^{N-2,N}$ and $Z^{N-1,N}$. The same procedure applies to $L^{N-j}$, $j=3,\ldots,N-1$.

To do so, let $\tilde{\mu}_t^{N-1}$ denote the joint law of $(L_t^{N-1},Z_t^{N-1,N})$ under $\mathcal{Q}^N$. Then it follows that 
\[
 \var_{Q_{N-1}}[L_t^{N-1}]
 = \Psi_t^{N-1,N}
 = \Psi^{N-1}(\tilde{\mu}_t^{N-1})
\]
where $\Psi^{N-1}$ is now a map $\Psi^{N-1}: \mathcal{P}_2(\mathbb{R}^2)\to\mathbb{R}$. 
Therefore, we can define a new volatility coefficient 
\begin{equation}
    \label{e:sigma-t}
    \tilde{\sigma}_{N-1}(t,\tilde{\mu}_t^{N-1})
    := \lambda^{N-1}(t, \Psi^{N-1}(\tilde{\mu}^{N-1}_t))
\end{equation}
which replaces $\sigma_{N-1}(t,\mu_t^{N-1})$ but, crucially, depends on the law $\tilde{\mu}_t^{N-1}$ under $\mathcal{Q}^N$.   

Substituting $\sigma_{N-1}(t,\mu_t^{N-1})$ by the coefficient $\tilde{\sigma}_{N-1}(t,\tilde{\mu}_t^{N-1})$, and using the relation \eqref{e:girs-i} between $W_t^{N}$ and $W_t^{N-1}$, yields
\begin{align*}
    \mathrm{d}Z^{N-2,N}_t 
    &= 
    Z^{N-2,N}_t \frac{\delta_{N-1} L_t^{N-1}}{1 + \delta_{N-1} L_t^{N-1}} \tilde{\sigma}_{N-1}(t,\tilde{\mu}_t^{N-1}) \left( \mathrm{d}W^{N}_t - \frac{\delta_N L_t^N}{1 + \delta_N L_t^N} \sigma_N(t,\mu_t^N)^{\top} \mathrm{d}t \right)  \\
    &\quad 
    + Z^{N-2,N}_t \frac{\delta_N L_t^{N}}{1 + \delta_N L_t^{N}} \sigma_N(t,\mu_t^{N})^{\top} \mathrm{d}W_t^{N} \\
    &\quad 
    + Z^{N-2,N}_t  \frac{\delta_N L_t^N}{1 + \delta_N L_t^N} \sigma_N(t,\mu_t^N)^{\top}\frac{\delta_{N-1} L_t^{N-1}}{1 + \delta_{N-1} L_t^{N-1}} \tilde{\sigma}_{N-1}(t,\tilde{\mu}_t^{N-1})  \mathrm{d}t \\
    &= 
    Z^{N-2,N}_t \frac{\delta_{N-1} L_t^{N-1}}{1 + \delta_{N-1} L_t^{N-1}} \tilde{\sigma}_{N-1}(t,\tilde{\mu}_t^{N-1})^{\top}  \mathrm{d}W^{N}_t \\
    &\quad 
    + Z^{N-2,N}_t \frac{\delta_N L_t^{N}}{1 + \delta_N L_t^{N}} \sigma_N(t,\mu_t^{N})^{\top} \mathrm{d}W_t^{N}.
\end{align*}
Therefore, we have found the SDE for $Z_t^{N-2,N}$ under the terminal measure and with coefficients which depend on the laws $\tilde{\mu}_t^{N-1}$ and $\mu_t^N = \tilde{\mu}_t^N$ also under the terminal measure. 

For $i<N-2$ we proceed analogously. Thus, let $\tilde{\mu}_t^i$ denote the joint law of $(L_t^i,Z_t^{i,N},\dots,Z_t^{N-1,N})$ under $\mathcal{Q}^N$. It follows that 
\[
 \var_{Q_i}[L_t^i]
 = \Psi_t^{i,N}
 = \Psi^i(\tilde{\mu}_t^i) \,,
\]
where $\Psi^i$ is now a map $\Psi^i: \mathcal{P}_2(\mathbb{R}^{N-i+1})\to\mathbb{R}$. 
Invoking \eqref{e:vol_gen}, we again define the new volatility coefficient 
\begin{equation}
    \label{e:sigma-t2}
    \tilde{\sigma}_i(t,\tilde{\mu}_t^i)
    := \lambda^i(t, \Psi^i(\tilde{\mu}^{i}_t))\,,
\end{equation}
which now depends on the law $\tilde{\mu}_t^i$ under the terminal measure $\mathcal{Q}^N$.

In total we arrive at the system
\begin{equation}
\left(
\begin{array}{c}
\mathrm{d}L_t^{i} \\[7pt]
\mathrm{d}Z^{N-1,N}_t \\[7pt]
\vdots \\[7pt]
\mathrm{d}Z^{i,N}_t
\end{array}
\right)
= 
\left(
\begin{array}{c}
L_t^i\left(  \tilde{\sigma}_i(t,\tilde{\mu}_t^i)^{\top} \mathrm{d}W^{N}_t - \sum_{k=i+1}^{N} \frac{\delta_k L_t^k}{1 +\delta_k L_t^k} \tilde{\sigma}_k(t,\tilde{\mu}_t^k)^{\top} \tilde{\sigma}_i(t,\tilde{\mu}_t^i) \mathrm{d}t \right) \\[7pt]
Z^{N-1,N}_t \frac{\delta_N L_t^{N}}{1 + \delta_N L_t^{N}} \sigma_N(t,\mu_t^{N})^{\top} \mathrm{d}W_t^{N} \\[7pt]
\vdots  \\[7pt]
Z^{i,N}_t \sum_{q=0}^{N-i-1} \frac{\delta_{N-q} L_t^{N-q}}{1 + \delta_{N-q} L_t^{N-q}} \tilde{\sigma}_{N-q}(t,\tilde{\mu}_t^{N-q})^{\top} \mathrm{d}W_t^{N} 
\end{array}
\right)\,.
\end{equation}

We now state our main result on the well-posedness of the mean field LIBOR market model (MF-LMM).

\begin{thm}[Existence of MF-LMM]\label{TH:TH1}
Let the volatility coefficients $\sigma_i: [0,t_{i-1}] \times \mathscr{P}_2(\RR) \to \RR^d$ be of the form
\begin{equation*}
    \sigma_i(t,\mu_{t}^{i}) 
    = 
    \sigma_i^{(1)}(t) 
    \exp\Big(
    -\max\left\{\var_{\mathcal{Q}^i}[L_t^{i}] - \tilde{\sigma},0\right\}
    \Big),
\end{equation*}
where $\mu^i$ denotes the law of $L^i$, and satisfy the following assumptions:
\begin{enumerate} 
    \item[(1)] There exists a constant $L>0$ such that
    \begin{equation*}
        | \sigma_i(t,\mu)| \leq L, \quad \forall \ t \in [0,t_{i-1}], \ \forall \ \mu \in \mathscr{P}_2(\RR).
    \end{equation*}
    \item[(2)] 
    For any $t\in[0,t_{i-1}]$, $\sigma(t,\cdot)$ satisfies assumption ({\bf A}$_{b\sigma}^1$) of Section~\ref{Sec:WEll} with a constant independent of~$t$. 
\end{enumerate}
Then the functions $\tilde{\sigma}_i: [0,t_{i-1}] \times \mathscr{P}_2(\RR^{N-i+1}) \to \RR$
defined by \eqref{e:sigma-t2} satisfy these assumptions as-well. 

Let $\mathcal{Q}^N$  be a probability measure and $W^N$ an adapted $d$-dimensional Brownian motion $W^{N}$. Then, the mean-field SDE 
\begin{equation}
\label{e:LIBORRates1}
    \mathrm{d}L_t^{N} 
    = L_t^{N}
        \tilde{\sigma}_N(t,\tilde{\mu}_t^{N})^{\top} 
        \mathrm{d}W^{N}_t , 
\end{equation}  
with $t\in[0,t_{N-1}]$, has a unique strong solution $L^N$.
Moreover, for $i < N$, the mean-field SDE 
\begin{align}
\label{e:LIBORRates}
\left(
\begin{matrix}
    \mathrm{d}L_t^{i} \\[7pt]
    \mathrm{d}Z^{N-1,N}_t \\[7pt]
    \vdots
    \\[7pt]
    \mathrm{d}Z^{i,N}_t
\end{matrix}
\right)
= 
\left(
\begin{matrix}
    L_t^i\left(  \tilde{\sigma}_i(t,\tilde{\mu}_t^i)^{\top} \mathrm{d}W^{N}_t - \sum_{k=i+1}^{N} \frac{\delta_k L_t^k}{1 +\delta_k L_t^k} \tilde{\sigma}_k(t,\tilde{\mu}_t^k)^{\top} \tilde{\sigma}_i(t,\tilde{\mu}_t^i) \mathrm{d}t \right) \\[7pt]
Z^{N-1,N}_t \frac{\delta_N L_t^{N}}{1 + \delta_N L_t^{N}} \tilde{\sigma}_N(t,\mu_t^{N})^{\top} \mathrm{d}W_t^{N} \\[7pt]
\vdots  \\[7pt]
Z^{i,N}_t \sum_{q=0}^{N-i-1} \frac{\delta_{N-q} L_t^{N-q}}{1 + \delta_{N-q} L_t^{N-q}} \tilde{\sigma}_{N-q}(t,\tilde{\mu}_t^{N-q})^{\top} \mathrm{d}W_t^{N} 
\end{matrix}
\right)
\end{align}
has a unique strong solution $(L^i,Z^{i,N}, \ldots, Z^{N-1,N})$,
where $\tilde{\mu}_t^i$ is the joint law of $(L_t^i,Z^{i,N}_t, \ldots, Z^{N-1,N}_t)$. 
Under the forward measure $\mathcal{Q}^i$, the martingale representation
\begin{equation}
\label{e:MRep}
    \mathrm{d}L_t^{i} 
    = L_t^{i} \sigma_i(t,\mu_t^{i})^{\top}\mathrm{d}W^{i}_t
\end{equation} 
holds, where $W^{i}$ is a $\mathcal{Q}^i$-Brownian motion in $\mathbb{R}^d$ and $\mu_t^{i}$ is the law of $L_t^{i}$ under $\mathcal{Q}^i$.
\end{thm}

\begin{proof}
We employ an inductive argument to prove well-posedness of the equations defined in \eqref{e:LIBORRates}. We start with $i=N$ and consider the mean-filed SDE
\begin{equation*}
\mathrm{d}L_t^{N} = L_t^{N} \sigma_N(t,\mu_{t}^{N})^{\top}\mathrm{d}W^{N}_t,
\end{equation*}
which, as shown in Theorem~\ref{TH:TH2}, has a unique strong (non-negative) solution, since it is a special case of the equation considered in Section \ref{Sec:WEll}. 

Assume that we have shown well-posedness for $i=k+1$. For $i=k$, we observe that the drift of
\begin{equation*}
\mathrm{d}L_t^{k} 
= 
 L_t^{k}\left( \tilde{\sigma}_k(t,\tilde{\mu}_t^{k})^{\top} \mathrm{d}W^{N}_t -  \sum_{j=k+1}^{N} \frac{\delta_j L_t^{j}}{1 +\delta_j L_t^{j}} \tilde{\sigma}_j(t,\tilde{\mu}_t^{j})^{\top} \tilde{\sigma}_k(t,\tilde{\mu}_t^{k}) \mathrm{d}t  \right), 
\end{equation*} 
depends on LIBOR rates with index larger than $k$, which, by induction hypothesis, have a  non-negative, unique strong solution. Furthermore, due to the boundedness of $\tilde{\sigma}$ (which follows from the boundedness of $\sigma)$, there exists a constant $C >0$ (independent of $t$ and the measure) such that
\begin{equation*}
\left| \sum_{j=k+1}^{N} \frac{\delta_j L_t^{j}}{1 +\delta_j L_t^{j}} \tilde{\sigma}_j(t,\tilde{\mu}_t^{j})^{\top} \tilde{\sigma}_k(t,\tilde{\mu}_t^{k}) \right| \leq C.
\end{equation*}
Therefore, Girsanov's theorem as stated in \cite[Lemma A.2]{BMBP18}, is applicable.\footnote{This also implies that all investigated change of measure intensities $Z^{i,N}$ for $i=1,\ldots,N-1$ are real martingales.}

Also the coefficients of each of the processes $(Z_t^{i})_{0 \leq t \leq T}$ only depend on LIBOR rates with index larger than $k$ and grow linearly in $Z^{i}$. Similar arguments to the ones employed in Section \ref{Sec:WEll} yield the claim, as the process $(L_t^i,Z^{i,N}_t, \ldots, Z^{N-1,N}_t)_{0 \leq t \leq T}$ can be identified with a $(N-i+1)$-dimensional process $(X_t)_{0 \leq t \leq T}$; See Remark~\ref{rem:ass}.
\end{proof}

\begin{remark}[Connection to Backward-Looking Rates (BLR)]
As noted in the introduction, our approach shares common features with the approach on backward-looking rates. This literature strand adapts classical forward-looking rates to the new concept of in arrear backward-looking rates while at the same time incorporating a deterministic (tenor-dependent) dampening effect at the respective tenor dates, see in eq.~(13) in \cite{LM19}, respectively eq.~(5) in \cite{LM19_}. We wish to stress that the MF-LMM approach and the BLR approach differ fundamentally, in particular:
\begin{itemize}
    \item \cite{LM19,LM19_} focus on extending traditional forward-looking rates (like the LIBOR) to encompass the new setting-in-arrears backward-looking rates, while staying compatible with these traditional rates. 
    \item The MF-LMM applies a distribution-dependent dampening in order to prevent explosion of the LIBOR rates, whereas the BLR approach uses a deterministic time-dependent dampening.
\end{itemize}
Additionally, as a result of the compatibility with traditional forward-looking (LIBOR) rates, our MF-LMM is as well consistent with backward-looking rates. We leave this topic for future research.
\end{remark}

\begin{remark}[Nonlinear diffusion in the sense of McKean] 
Equation \eqref{e:LIBORRates} is, viewed individually, a mean-field SDE only if one considers the input from other distributions as an additional time-dependence and allows for random coefficients. Indeed, for $i<N$, the coefficients in \eqref{e:LIBORRates} depend on distributions $\tilde{\mu}^j$ and rates $L^j$ with $i<j\le N$.  
By contrast, equation~\eqref{e:MRep} is a mean-field SDE for each $i$.  
Moreover, equations~\eqref{e:LIBORRates1} and \eqref{e:LIBORRates} can be reformulated as components of the mean-field SDE \eqref{e:MF} (with non-random coefficients):

With the same assumptions and notation as in Theorem~\ref{TH:TH1}, the $\mathbb{R}^{2N-1}$-valued process 
\[
 M = (L^1,\dots,L^N,Z^{1,N},\dots,Z^{N-1,N})\,,
\]
satisfies the mean-field equation
\begin{equation}
\label{e:MF}
     \mathrm{d}M_t
     = \tilde{b}(t,M_t,\tilde{\mu}_t)\di t
       + \tilde{\sigma}(t,\tilde{\mu}_t)^{\top} 
        \mathrm{d}W^{N}_t \,,
\end{equation}
where $\tilde{\mu}_t$ is the law of $M_t$ under $\mathcal{Q}_N$. Then, the components of this equation are defined as follows:

For the collection $(L^j,Z^{j,N},\ldots,Z^{N-1,N})$ for $j=1,\ldots,N-1$ let $m_{(L^j,Z^{j,N},\ldots,Z^{N-1,N})}: \mathscr{P}_2(\RR^{2N-1}) \to\mathscr{P}_2(\RR^{N-j+1})$ be the projection onto the corresponding marginal which corresponds to the joint law of $(L^j,Z^{j,N},\ldots,Z^{N-1,N})$ under $\mathcal{Q}^N$.
Consider
\begin{equation}
\label{e:LR1}
    \mathrm{d}L_t^{N} = L_t^{N}
        \tilde{\sigma}_N(t,\tilde{\mu}_t^{N})^{\top} 
        \mathrm{d}W^{N}_t , 
\end{equation}  
where $0 \leq t < t_{N-1}$ and $\tilde{\mu}_t^N = \mu_t^N = m_{L^N}(\tilde{\mu}_t)$ is the law of $L^{N}_t$.
Furthermore, for $i < N$, consider
\begin{align}
\label{e:LR2}
    \mathrm{d}L_t^{i}
    &=
     L_t^i\left(  \tilde{\sigma}_i(t,\tilde{\mu}_t^i)^{\top} \mathrm{d}W^{N}_t - \sum_{k=i+1}^{N} \frac{\delta_k L_t^k}{1 +\delta_k L_t^k} \tilde{\sigma}_k(t,\tilde{\mu}_t^k)^{\top} \tilde{\sigma}_i(t,\tilde{\mu}_t^i) \mathrm{d}t \right)\,,
    \\
\label{e:LR3}
    \mathrm{d}Z^{i,N}_t
    &=
    Z^{i,N}_t \sum_{q=0}^{N-i-1} \frac{\delta_{N-q} L_t^{N-q}}{1 + \delta_{N-q} L_t^{N-q}} \tilde{\sigma}_{N-q}(t,\tilde{\mu}_t^{N-q})^{\top} \mathrm{d}W_t^{N} \,,
\end{align}
where 
$\tilde{\mu}_t^i 
= m_{(L^i,Z^{i,N},\ldots,Z^{N-1,N})}(\tilde{\mu}_t)$ is the joint law of $(L_t^i,Z^{i,N}_t, \ldots, Z^{N-1,N}_t)$.
\end{remark}

\section{Practical aspects of the MF-LMM}\label{SEC:Cali}

\subsection{A mean-field cap formula}

In order to show that our framework is consistent with classical LIBOR market models, we also provide a Black '76-type cap pricing formula for a given measure flow along the lines of \cite{Black76}, cast into our mean-field framework:
\begin{thm}\label{thm:cap}
Let $t_1 < \ldots < t_N$ be a given tenor structure. We assume that the LIBOR rates $L_t^{i}$, for $t < t_{i-1}$ on the time-interval $[t_{i-1},t_i]$, associated to this tenor structure follow the evolution specified by equation (\ref{e:LMM}) and that $\{\mu_s^i\,\vert\,0\leq s\leq t_{i-1}\}$ are given distributions in $\mathcal{P}_2(\R)$ (for each $s$ and $i$).
In the sequel, $K$ and $V$ will denote the strike and nominal value, respectively. Then, we have that:
\begin{enumerate}
\item The price $C_i(t,\sigma_i(\cdot,\mu_{\cdot}^{i}))$ of a caplet with expiry date $t_i$ and payoff $\delta_i \cdot (L_{t_{i-1}}^{i} -K)^{+}$ is determined by
\begin{align*}
& C_i(t,\sigma_i(\cdot,\mu_{\cdot}^{i})) = \delta_i P(t,t_i) \left[ L_t^{i} \Phi(d_t^{1}) -K \Phi(d_t^{2})  \right], \\
& d_t^{1} = \frac{\log\left( \frac{L_{t}^{i}}{K} \right) + \frac{1}{2} (\bar{\sigma}_t^{i})^2}{\bar{\sigma}_t^{i}}, \quad d_t^{2} = d_t^{1} - \bar{\sigma}_t^{i}, \\
& (\bar{\sigma}_t^{i})^2 = \int_{t}^{t_{i-1}} \left(\sigma_i(s,\mu_{s}^{i}) \right)^2 \mathrm{d}s.
\end{align*}
\item The price $Cap_{FL}(t;V,K)$ of a cap, consisting of caplets with expiry dates $t_1 < \ldots < t_N$ such that $t < t_1$, in the MF-LMM is given by
\begin{equation*}
Cap_{FL}(t;V,K) = V \sum_{i=1}^{N} C_i(t,\sigma_i(\cdot,\mu_{\cdot}^{i})). 
\end{equation*}
\end{enumerate}
\end{thm}
\begin{proof}
The second item is a straightforward consequence of the first, as the individual rates are independent. Since the $i$-th LIBOR rate is modeled under $\mathcal{Q}^i$, we obtain 
\begin{align*}
C_i(t,\sigma_i(t,\mu_{t}^{i})) &= \delta_i P(t,t_i) \mathbb{E}^{\mathcal{Q}^i}\left(L_{t_{i-1}}^{i} - K \right)^{+} \\
& = \delta_i P(t,t_i) 
\mathbb{E}^{\mathcal{Q}^i}
\left(  L_{t_{i-1}}^{i} \mathrm{I}_{\{L_{t_{i-1}}^{i} >K\}}  \right) - \delta_i\, K \mathcal{Q}^i \left( L_{t_{i-1}}^{i} > K \right).
\end{align*} 
Due to (\ref{e:LMM}), we have that $L_{t_{i-1}}^{i}$ under $\mathcal{Q}^i$ is log-normal distributed with 
\begin{align*}
\log \left( L_{t_{i-1}}^{i} \right) \sim \mathcal{N} \left( \log \left( L_{t}^{i} \right) - \frac{1}{2} \int_{t}^{t_{i-1}} \left(\sigma_i(s,\mu_{s}^{i}) \right)^2 \mathrm{d}s, \int_{t}^{t_{i-1}} \left(\sigma_i(s,\mu_{s}^{i}) \right)^2 \mathrm{d}s \right).
\end{align*}
The remaining part of the proof follows now standard arguments, as in the derivation for the Black-Scholes model.  
\end{proof}
This corresponds now exactly to the market price formula in a Black-Scholes sense with an averaged volatility. In contrast to the classical LIBOR market model, we now average over measure-dependent functions for the volatility instead of purely deterministic functions. Note that this formula is also the theoretical foundation of the suggested calibration procedure in the following section. 

\subsection{Mean-field calibration}\label{sec:mfCal}

The LIBOR rates 
\begin{align*}
\mathrm{d}L_t^{i} = L_t^{i} \sigma_i(t,\mu_{t}^{i})\mathrm{d}W^{i}_t, 
\end{align*}
where $\mu_{t}^{i} = \text{Law}(L_t^{i})$ and the law is considered under the measure $\mathcal{Q}^i$, will be used for the calibration of the coefficients $\sigma_i$. Recall that the dependence of $\sigma_i$ on the measure is already specified in Theorem \ref{TH:TH1}.  

To cast the problem into a well-known setting, one possibility is to use the auxiliary model
\begin{align*}
\mathrm{d}\tilde{L}_t^{i} = \tilde{L}_t^{i} \sigma_i^{(1)}(t)\mathrm{d}W^{i}_t\,, 
\end{align*} 
which focuses on the deterministic component of \eqref{eq:diffusion}. Then one can calibrate $\sigma_i^{(1)}$ to market data using classical cap prices. Parametric approaches for the calibration of $\sigma_i^{(1)}$, explained in \cite{Brigo}, are constructed in a way such that $\sigma_i^{(1)}$ increases as $t$ approaches the time to maturity.
Clearly, a modified model with added measure dependent coefficients is not a-priori market consistent, but the numerical results from Section \ref{sec:numerics} demonstrate that for long maturities consistency is achieved. Since avoidance of explosive paths is crucial for efficient valuation of long term contracts, this approach with a distinct damping factor seems to be quite meaningful for practical considerations.

However, with the help of Theorem~\ref{thm:cap} and an iterative procedure we can also directly calibrate the mean-field LMM. To achieve this, we propose the following procedure: At first we observe that, for given distributions $\{\mu_s^i\,\vert\,0\leq s\leq t_{i-1}\}$ (which correspond to estimates for the variance of $L_s^{i}$ at time $s$), given a perfect calibration we would have that
\begin{align*}
\left(\hat{\sigma}^{\text{market}}_i \right)^2= \int_{0}^{t_{i-1}} \left(\sigma_i(s,\mu_{s}^{i}) \right)^2 \mathrm{d}s \,,
\end{align*}
where $\hat{\sigma}^{\text{market}}_i$ is the quoted implied volatility of a caplet with expiry date $t_i$.

The implemented method is inspired by the Picard-type iteration employed in the proof of Theorem \ref{TH:TH2}. We start with a given initial variance function $v^{(0)}:[0,t_{i-1}]\to\RR^+$ and use it as a substitute for the yet unknown variance of $L^i$, i.e.,
$v^{(0)}(s)\widehat{=}\var_{\mathcal{Q}^i}(L_s^i)$. Then, we further approximate 
\begin{align}\label{eq:sigmaapprox}
\left(\hat{\sigma}^{\text{market}}_i \right)^2 &=  \int_{0}^{t_{i-1}} \left(\sigma_i(s,\mu_{s}^{i}) \right)^2 \mathrm{d}s \notag \\
& \approx\sum_{j=1}^{J}  \left(\sigma_{i}^{(1)}(s_j) \right)^2 \exp\left(-2 \max\{v^{(0)}(s_j)-\tilde{\sigma},0\}/\tilde{\sigma} \right) \left( s_j - s_{j-1} \right),
\end{align}
with $0=s_0<\ldots<s_J=t_{i-1}$ for some $J\in\N$.\\
At first we focus on the deterministic component $\sigma_i^{(1)}$ of the LIBOR rates' volatility coefficient using \eqref{eq:sigmaapprox}. If we assume a parametric form of $\sigma_i^{(1)}$, we compute a first estimate of its parameters by determining: 
\begin{align*}
\mbox{argmin}\,\Big |{\left(\hat{\sigma}^{\text{market}}_i \right)^2 -  \sum_{j=1}^{J}  \left(\sigma_i^{(1)}(s_j) \right)^2 \exp\left(-2 \max\{v^{(0)}(s_j)-\tilde{\sigma},0\}/\tilde{\sigma} \right) \left( s_j - s_{j-1} \right)}\Big |_2.
\end{align*}
For example, if we use $\sigma_i^{(1)}(t)=g(t_{i-1}-t)$ with
\begin{align}\label{eq:sigma_parametric}
g(\tau)=(a+b\tau)e^{-c\tau}+d,
\end{align}
the above minimization is with respect to $\{a,b,c,d\}$. We denote the resulting first estimator by $\hat{\sigma}_i^{(1,1)}$. In a next step we need to update the variance $v^{(0)}$. 
Therefore, we run Monte Carlo simulations to generate $M\in\NN$ independent paths of the corresponding process,
\begin{align*}
\mathrm{d}L_t^{i}= L_t^{i}\hat{\sigma}_i^{(1,1)}(t)\exp\left(-\max \{v^{(0)}(t)-\tilde{\sigma},0\}/\tilde{\sigma}\right) \, \mathrm{d}W_t^i,
\end{align*}
up to the terminal time $t_{i-1}$. We derive an approximated update for the variance function  $v^{(1)}$ from the associated empirical distribution. This procedure can now be immediately iterated. Replace $v^{(0)}$ by $v^{(1)}$ for computing $\hat{\sigma}_i^{(1,2)}$ in a first step and continue until the norm $|\hat{\sigma}_i^{(1,l+1)}-\hat{\sigma}_i^{(1,l)}|$ becomes acceptably small for some $l\in\NN$.
\\
\\
In the following numerical example we demonstrate the feasibility of the above described approach. We focus on $t_i=20$ and use an equidistant grid $s_j-s_{j-1}=h=\frac{1}{30}$.
For the deterministic component of the volatility we use functions of the type $g(\tau)=(a+b\,\tau)e^{-c\,\tau}+d$, such that for $L^i$ we have $\sigma_i^{(1)}(t)=g(t_{i-1}-t)$.
Specifically, for our \emph{toy} example we choose $a=0.14,\,b=0.01,\,c=0.05,\,d=0.2$ and use a number of paths ($M=10^5$) to approximate $\mathbb{E}^{\mathcal{Q}_i}[\max\{L_{t_i-h}-K),0\}]$. The corresponding results for the calibrated deterministic model are stated in Table \ref{tab:Iterations} for $k=0$ iteration steps. 

Using the results from Theorem \ref{thm:cap} with $\delta_i=1$, strike $K=L_0^i=0.02$ and implicitly setting $P(0,t_i)=1$, we can compute an associated \emph{quoted implied} volatility $\hat{\sigma}_i^{market}=1.55$. For the damping volatility we choose $\tilde{\sigma}=\frac{1.55}{20}$. This choice is motivated by uniformly distributing $\hat{\sigma}_i^{market}$ over the considered period of time.\\
Table \ref{tab:Iterations} at $k=6$ collects the results after 6 iteration steps. In general, one can say that the estimated parameters of $g$ stabilize very quickly. Thus the stated relative error is mainly due to the MC-method (remember that the stepsize is $h=\frac{1}{30}$).
\begin{center}
\begin{table}
\begin{tabular}{ c| c| c|c }
k & CI for \emph{caplet} price  & Parameters in $g(\cdot)$ & Relative error\\
\hline
0 & (0.0109399,{\bf{0.0112332}},0.0115265) & $\{0.14,0.01,0.05,0.2\}$ & - \\
6 & (0.0110543,{\bf{0.0113342}},0.0116141) & $\{2.08184,0.878775,3.89368,0.262653\}$ & 0.00899268
\end{tabular}
\caption{\label{tab:Iterations} Approximated caplet prices and estimated parameters.}
\end{table}
\end{center}
The effect of the variance dependent term can nicely be seen in Figures \ref{fig:Variances} and \ref{fig:Tails}. In Figure \ref{fig:Variances} the red line plots $\var_{Q_i}[L_t^{i,1}]$ resulting from the $M$ paths of the classical LMM as a function of $t\in [0,t_i-h)$. The blue line presents $\var_{Q_i}[L_t^{i,6}]$ computed from the paths of the 6th iteration step for $t\in [0,t_i-h)$.
\begin{figure}
\centering
\includegraphics[width=7cm]{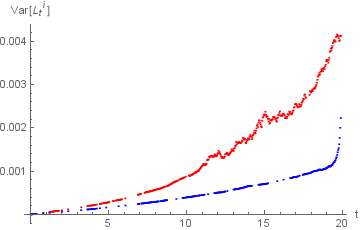}
\caption{Resulting variances of classical and mean-field LMM.}
\label{fig:Variances}
\end{figure}
In Figure \ref{fig:Tails} we depict the tail of the empirical distributions of $L_t^{i}$ and $L_t^{i,6}$ for $t=15$. Again the red line corresponds to the case of a deterministic volatility component and the blue one to the mean-field situation. One can nicely see that the un-damped variant features more probability mass for rates above 40\% of interest.
\begin{figure}
\centering
\includegraphics[width=7cm]{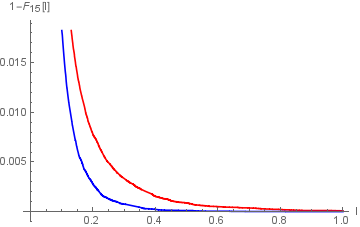}
\caption{Tail of the distribution of $L_{15}^i$ classical and mean-field LMM.}
\label{fig:Tails}
\end{figure}

\subsection{Conversion to the spot measure}\label{sec:spot}
Since existence and uniqueness of the mean-field system \eqref{e:LMM}, \eqref{eq:diffusion} are proved,
all the usual transformation rules of LIBOR market models apply. If a measure change, as below to the spot measure, is carried out it only remains to calculate the variances in \eqref{e:vol_gen} with respect to the new measure. 

Let $\mathcal{Q}^*$ be the spot measure and $W^*$ the corresponding $d$-dimensional Brownian motion as in, e.g., Section~11.4 of \cite{Fil}. 
The Girsanov transformation, keeping in mind the relation $ \sigma_k(t,\mu_t^{k}) = \lambda^k(t,\var_{\mathcal{Q}^k}(L_t^k))$, yields
\begin{equation}
\label{e:spot1}
    \di L_t^m
    = 
    L_t^m\Big(
        \sum_{k=\eta(t)}^m\frac{\delta_k L_t^k}{\delta_k L_t^k + 1}
        \lambda^k(t,\var_{\mathcal{Q}^k}(L_t^k))^{\top}
        \lambda^m(t,\var_{\mathcal{Q}^m}(L_t^m)) \di t
        +
        \lambda^m(t,\var_{\mathcal{Q}^m}(L_t^m))^{\top} \di W_t^*
    \Big)
\end{equation}
where the right-continuous function $\eta: [0,t_{M-1}]\to\{1,\dots,M\}$ is such that
\begin{equation}
   \label{e:eta}
    t_{\eta(t)-1}\le t < t_{\eta(t)}.
\end{equation}

However, the variances continue to depend on the forward measures. A calculation analogous to \eqref{e:V-trnsf} implies
\begin{equation}
\label{e:Phi_t}
    \var_{\mathcal{Q}^m}[L_t^m]
    = 
   \mathbb  E_{\mathcal{Q}^*}\Big[\Big(
    L_t^m - \mathbb E_{\mathcal{Q}^*}[L_t^m Y_t^m]
    \Big)^2 Y_t^m \Big]
    =: \Psi_t^{m,*} 
\end{equation}
where 
\begin{equation*}
    Y_t^m
    :=
    \frac{\mathrm{d}\mathcal{Q}^m}{\mathrm{d}\mathcal{Q}^*}\Big|_{\mathcal{F}_t}
    =
   \mathbb  E_{\mathcal{Q}^*}\Big[\frac{\mathrm{d}\mathcal{Q}^m}{\mathrm{d}\mathcal{Q}^*}\Big|\mathcal{F}_t\Big].
\end{equation*}


The expressions $\lambda^k(t,\Psi_t^{k,*})$ are the volatility coefficients in the spot formulation \eqref{e:spot1} with respect to the joint law, $\mu_t^{k,*}$, of $(L_{t}^k, Y_t^k)$ under $\mathcal{Q}^*$. We can also write $\lambda^k(t,\Psi_t^{k,*}) = \sigma_k^*(t,\mu_t^{k,*})$ to emphasize the dependence on $\mu_t^{k,*}$.

\subsubsection{Continuous-time formulation}
The process $Y^m$ can be expressed as a stochastic exponential (\cite[Equ.~(7.1)]{Fil})
\begin{equation}
\label{e:sto_exp}
    Y_t^m 
    =
    \mathcal{E}_t\left(
    - \int_0^\cdot \sum_{k=\eta(s)}^m\frac{\delta_k L_s^k}{\delta_k L_s^k + 1}
        \lambda^k(s,\Psi_s^{k,*})^{\top}
         \mathrm{d} W^*_s 
    \right).
\end{equation}
With \eqref{e:Phi_t} for the variance in terms of the joint law of $L_t$ and $Y_t$ under $\mathcal{Q}^*$ this means that \eqref{e:spot1} can be transformed into a mean field system along the lines of Theorem~\ref{TH:TH1}, where the $Y_t^m$ play now the roles of the $Z_t^m$:
\begin{align} 
\label{e:mfLY1}
    dL_t^m
    &=
    L_t^m\left(
        \sum_{k=\eta(t)}^m\frac{\delta_k L_t^k}{\delta_k L_t^k + 1}
        \lambda^k(t,\Psi_t^{k,*})^{\top}
        \lambda^m(t,\Psi_t^{m,*}) d t
        +
        \lambda^m(t,\Psi_t^{m,*})^{\top} \mathrm{d} W^*_t
    \right)\,,
    \\
\label{e:mfLY2}
    d Y_t^m
    &= 
    - Y_t^m \sum_{k=\eta(t)}^m\frac{\delta_k L_t^k}{\delta_k L_t^k + 1}
        \lambda^k(t,\Psi_t^{k,*})^{\top}
        \mathrm{d} W_t^*\,,
\end{align}
with $Y_0^m = 1$. 
Equations~\eqref{e:mfLY1} and \eqref{e:mfLY2} are again a mean field system of SDEs since the $\Psi_t^{k,*}$ depend on the joint law of $L_t^k$ and $Y_t^k$ under $\mathcal{Q}^*$. 
 
\subsubsection{Projection along tenor dates}
If $t=t_j$ is a tenor date, then the conditional expectation can be expressed as
\begin{equation}
    Y_{t_j}^m
    =
    B^*(t_j)^{-1}\frac{P(j,m)}{P(0,m)}\,,
\end{equation}
(see \cite[Sec.~7.1]{Fil}) where 
\begin{equation}
    B^*(t_j)
    = (1+\delta_{j-1} L_{t_{j-1}}^{j-1})B^*(t_{j-1}),
    \quad
    B^*(t_0) = 1\,,
\end{equation}
is the implied money market account (i.e., the numeraire) and 
\begin{equation}\label{eq:bond}
    P(j,m)
    =
    \Pi_{l=j}^{m-1}(1+\delta_l L_{t_j}^l)^{-1}\,,
\end{equation}
is the time $t_j$-value of one unit of currency paid at $t_m$. 

With
\begin{equation}
    \label{e:phi1}
    \Psi_j^{m,*}
    := 
    E_{\mathcal{Q}^*}\left[\left(
    L_{t_j}^m - E_{\mathcal{Q}^*}\left[L_{t_j}^m  B^*(t_j)^{-1}\frac{P(j,m)}{P(0,m)}\right]
    \right)^2  B^*(t_j)^{-1}\frac{P(j,m)}{P(0,m)} \right]\,
\end{equation}    
it follows that the evolution along the tenor dates of the mean field system \eqref{e:mfLY1}-\eqref{e:mfLY2} is given  by
\begin{equation}
\label{e:spot_phi}
    d L_{t_j}^m
    = 
    L_{t_j}^m\left(
        \sum_{k=j+1}^m\frac{\delta_k L_{t_j}^k}{\delta_k L_{t_j}^k + 1}
        \lambda^k(t_j,\Psi_j^{k,*})^{\top}
        \lambda^m(t_j,\Psi_j^{m,*}) \di t
        +
        \lambda^m(t_j,\Psi_j^{m,*})^{\top} \mathrm{d} W^*_t\,
    \right),
\end{equation}
since $\eta(t_j)=j+1$.

\subsubsection{Exogenous mean-field dynamics}
The discrete spot measure formulation \eqref{e:spot_phi} is the basis for the numerical scheme in Section~\ref{sec:numerics}.
Therefore, the volatility structure $\lambda^m(t,\Psi_t^{m,*})$ in \eqref{e:mfLY1}-\eqref{e:mfLY2} has to be specified.
To this end, and following the ideas of Section~\ref{sec:vol_str}, we split the volatility structure as 
\begin{equation}
\label{e:split}
    \lambda^m(t,\Psi_t^{m,*})
    =
    \sigma^{(1)}_m(t) \lambda^{\textup{mf}}(\Psi_t^{m,*})\,,
\end{equation}
where $\sigma^{(1)}_m(t)$ is a deterministic volatility specification and $\lambda^{\textup{mf}}(\Psi_j^{m,*})$ depends  on the distribution of $(L_t^m,Y_t^m)$ under the spot measure. We remark that $\lambda^{\textup{mf}}$ is assumed to be time-homogeneous, i.e. there is no explicit dependence on time.  In principle, $\lambda^{\textup{mf}}$ could also depend on the maturity $m$ but we will not need this extra degree of freedom.
Examples of possible choices for $\sigma^{(1)}_m(t)$ can e.g. be found in \cite{Brigo}.
The choices for the numerical study are presented in Section~\ref{sec:numerics} below. 

\begin{remark}
This approach seems to be promising when evaluating long-term guarantees as a part of life insurance or pension contracts. Here, one obtains the $\sigma_m^{(1)}$ by calibrating a classical LMM to market data and uses \eqref{e:split} with a variance dependent dampening factor in the internal Monte Carlo procedure. At first sight, this has the consequence that the valuation principle is not market-consistent, but the involved routine is more stable. Moreover, the numerical results from Section~\ref{sec:numerics} demonstrate in particular that for long maturities the difference is negligible.
\end{remark}

\section{Monte Carlo simulation of MF-LMM}
\label{sec:numerics}

The numerical implementation of the MF-LMM is based on an Euler-Maruyama scheme for \eqref{e:spot_phi} together with a specification of the splitting assumption \eqref{e:split}.
Making use of \eqref{e:spot_phi} implies that time steps equal tenor dates. This has the practical advantage that the
empirical variance \eqref{e:Phi_t} can be calculated without the necessity of simulating the process $Y^m$, defined in \eqref{e:mfLY2}. Moreover, this choice is compatible with industry practice for valuation of long term guarantees where the simulation time may be of the order of $100$ years and yearly time steps are generally used.  

For our numerical simulations, we consider the risk free term structure that has been provided by EIOPA at year-end 2020 (\cite{EIOPA_curve}). 
This interest rate curve is used by European insurance companies for the calculation of technical provisions.

In order to deal with negative interest rates, we apply a displacement factor $\alpha\in\mathbb{R}_{\ge0}$
(see \cite{diffusion} and \cite[p.~471]{Brigo}). 

\subsection{Euler-Maruyama discretization scheme}
\label{sec:euler}
Let $M$ be a positive integer and
consider a discrete time grid $0, 1, 2, \dots, M$ consisting of yearly time steps. Elements in the time grid will be denoted by $t_n$.

For numerical purposes the mean field equation \eqref{e:spot_phi} is approximated by an interacting particle system (IPS), compare equation \eqref{eq:euler1}. Each `particle' in this IPS corresponds to an interest rate curve that interacts with all other `particles' via a Monte Carlo approximation of \eqref{e:phi1}. If the number of particles, $P$, is sufficiently large we obtain a numerical approximation of the mean-field model. 

Within the IPS, the simulated stochastic processes evolve according to $P$ independent copies of a $d$-dimensional Brownian motion 
$(W^{p})_{1 \leq p \leq P}$ with increments 
$\Delta W^{p}_{t_n} := W^{p}_{t_{n+1}} - W^{p}_{t_n}$. 
For a given maturity $m$, all `particles' $L^{m,p}$ start from the same initial interest rate curve $L_0^m$.
In order to improve the numerical stability of the model, we will work on a logarithmic scale within the simulation. 
The numerical scheme is therefore as follows: at $t_n$ and for $1 \leq p \leq P$, the subsequent LIBOR rate is updated using the discretization
\begin{align}
\log(L_{t_{n+1}}^{m,p}+\alpha) 
    &= \log(L_{t_n}^{m,p}+\alpha) 
    +  \lambda^m(t_n,\Psi_{t_n}^{m,*})^{\top} 
        \sum_{k=\eta(t_n)}^{m} \frac{ L_{t_n}^{k,p}+\alpha}{1 + L_{t_n}^{k,p} + \alpha} \lambda^k(t_n,\Psi_{t_n}^{k,*})    \nonumber\\
&\phantom{==} 
    - \frac{1}{2} | \lambda^m(t_n,\Psi_{t_n}^{m,*}) |^2 
    + \lambda^m(t_n,\Psi_{t_n}^{m,*})^{\top} \Delta W^{p}_{t_n} , 
\label{eq:euler1}
\end{align}
where the right continuous function $\eta(t_n) = n+1 = t_{n+1}$ is defined in \eqref{e:eta}. 
The dependence on the joint law of the forward rates in \eqref{e:phi1} at time $t_n$ and maturity $m$ is thus realized as 
\begin{equation}
    \label{e:Psi_sim}
 \Psi_{t_n}^{m,*}  
 =  
 \frac{1}{P}\sum_{p=1}^P 
 \left[\left(
    L_{t_j}^{m,p} - 
    \frac{1}{P}\sum_{q=1}^P
    \left[L_{t_j}^{m,q}  B_q^*(t_j)^{-1}\frac{P_q(j,m)}{P_q(0,m)}\right]
    \right)^2  
    B_p^*(t_j)^{-1}\frac{P_p(j,m)}{P_p(0,m)} \right]\,,
\end{equation}
where $P_p(n,i)$ refers to the time $t_n$-value of one unit of currency paid at $t_i$ in the $p$-th particle, $1 \leq p \leq P$. 

By convention the dimension, $d$, of the Brownian increments, $\Delta W_{t_n}^p$, shall be equal to the dimension of the vectors $\lambda^m$.

\subsection{Volatility structure}\label{sec:VS2}
The classical LIBOR market model (without any mean-field interaction) is immediately realized as a special case of our approach \eqref{eq:euler1} by setting $\lambda^{\textup{mf}}=1$ in \eqref{e:split}. 
For $\sigma_m^{(1)}(t)$ 
we consider the parametric volatility structure  given by
\begin{equation}\label{eq:vola1}
    \lambda^m(t)
    =
    \sigma_m^{(1)}(t) 
    := 
    \left(\Big(a(t_{m-1} - t) + d\Big) e^{-b(t_{m-1} -t)} + c\right)
    \left(
    \begin{matrix}
        \cos\theta_m\\
        \sin\theta_m
    \end{matrix}
    \right), 
    \quad t \le t_{m-1} \,,
\end{equation}
where the $\theta_m$ are angles which depend on the maturity but not on time.
Thus, in this case, the dimension of the Brownian increment in \eqref{eq:euler1} is $d=2$. 
This choice provides a hump-shaped structure for instantaneous volatility of the LIBOR rate $L^m$ as a function of the time to maturity. See \cite{Brigo,Rebonato} for further background and an economic interpretation.

The subsequently presented results are obtained with respect to the year-end 2020 (without the so-called volatility adjustment) risk-free EIOPA interest rate curve (\cite{EIOPA_curve}) with a projection horizon of $M=50$.
 
The displacement factor is fixed as 
\begin{equation}
    \label{e:dispF}
    \alpha = 1\,\%.
\end{equation}
We consider two sets of parameters for the hump shaped volatility curve: 
\begin{align}
    \textup{RMW parameters: }\quad
    &
    a = 0.07,\quad
    b = 0.2, \quad
    c = 0.6, \quad
    d = 0.075  \label{e:paraRMW} \\
    \textup{Excited parameters: }\quad
    &a = 0.01,\quad
    b = 0.05, \quad
    c = 0.2, \quad
    d = 0.14 
    \label{e:para}
\end{align}
The values \eqref{e:paraRMW} are taken from the textbook \cite[page~13]{RMW} where these are interpreted as representing a `normal' state of the volatility structure. The parameters \eqref{e:para} are chosen specifically to represent an excited state of the market with increased volatility, see Figure~\ref{fig:vol_cur}. This is where the blow-up problem is most pronounced whence the mean-field interaction has the strongest (and graphically most visible) effect.

\begin{figure}
    \centering
    \includegraphics[width=7cm]{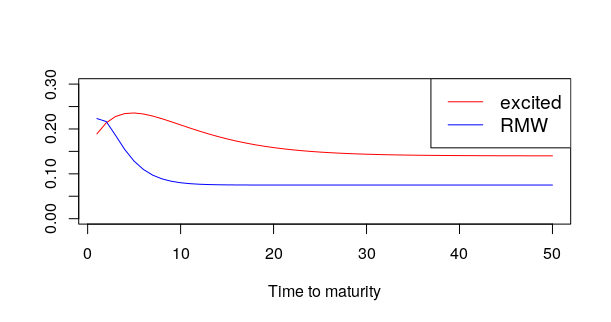}
    \caption{Volatility curves representing the scalar part of \eqref{eq:vola1} with respect to time to maturity $\tau_m=M-t_m$ and parameters \eqref{e:paraRMW}-\eqref{e:para}.}
    \label{fig:vol_cur}
\end{figure}
 
The angles, $\theta_m$ are chosen as in Figure~\ref{fig:theta} to represent a generic and economically plausible correlation structure~\eqref{e:corrM}.  
\begin{figure}
    \centering
    \includegraphics[width=7cm]{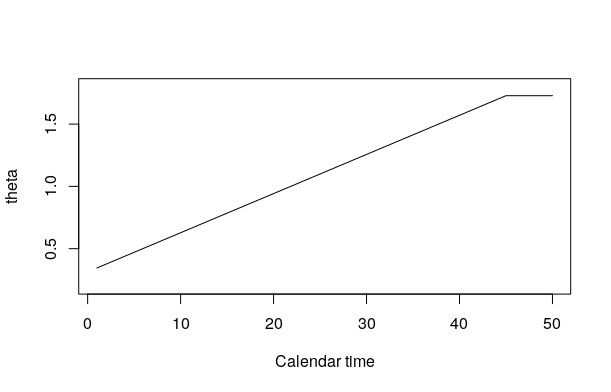}
    \caption{Choice of angles, $\theta_m$, as a function of time indexed by $m$.}
    \label{fig:theta}
\end{figure}
 
The choice depicted in Figure~\ref{fig:theta} yields the following correlation structure $\cos(\theta_m-\theta_n)$,  where indices $m,n$ are in $\{1,6,11,16,\dots,46\}$:
\begin{equation}
\label{e:corrM}
\begin{bmatrix}
  1.00 & 0.99 & 0.95 & 0.89 & 0.81 & 0.71 & 0.59 & 0.45 & 0.31 & 0.19 \\ 
  0.99 & 1.00 & 0.99 & 0.95 & 0.89 & 0.81 & 0.71 & 0.59 & 0.45 & 0.34 \\ 
  0.95 & 0.99 & 1.00 & 0.99 & 0.95 & 0.89 & 0.81 & 0.71 & 0.59 & 0.48 \\ 
  0.89 & 0.95 & 0.99 & 1.00 & 0.99 & 0.95 & 0.89 & 0.81 & 0.71 & 0.61 \\ 
  0.81 & 0.89 & 0.95 & 0.99 & 1.00 & 0.99 & 0.95 & 0.89 & 0.81 & 0.73 \\ 
  0.71 & 0.81 & 0.89 & 0.95 & 0.99 & 1.00 & 0.99 & 0.95 & 0.89 & 0.83 \\ 
  0.59 & 0.71 & 0.81 & 0.89 & 0.95 & 0.99 & 1.00 & 0.99 & 0.95 & 0.90 \\ 
  0.45 & 0.59 & 0.71 & 0.81 & 0.89 & 0.95 & 0.99 & 1.00 & 0.99 & 0.96 \\ 
  0.31 & 0.45 & 0.59 & 0.71 & 0.81 & 0.89 & 0.95 & 0.99 & 1.00 & 0.99 \\ 
  0.19 & 0.34 & 0.48 & 0.61 & 0.73 & 0.83 & 0.90 & 0.96 & 0.99 & 1.00 \\ 
   \end{bmatrix}
\end{equation}
In the following, we compare four simulation methods differing in their volatility structure. We refer to the classical model without mean-field dependence as \emph{VolSwi2} and the method with mean-field taming as \emph{VolSwi25} (Section~\ref{sec:VS25}). We furthermore present simulation methods with correlation assumptions based on economic considerations including anti-correlation (\emph{VolSwi4} in Section~\ref{sec:VS4}) and decorrelation of interest rates (\emph{VolSwi6} in Section~\ref{sec:VS6}) in order to deal with "exploding" rates". In this context we define blow-up, or explosion, as the occurrence of a significant number (i.e., more than $1\%$) of scenarios beyond a certain threshold (i.e., $50\%$ interest) at a given time. This can be tested graphically by looking at the histograms in Figures~\ref{fig:hist_RMW} and \ref{fig:hist} of $L^1$ at times $10$, $20$, $30$ and $40$, or at the excess plots in Figure~\ref{fig:exc}. 

\subsection{Mean-field taming beyond threshold}\label{sec:VS25}
One possibility of mitigating explosion is to include a taming factor which depends on the observed scenario variance, $\Psi_t^{m,*}$, at time $t$ under the spot measure. Thus we choose a variance threshold $\tilde\sigma$ and define
\begin{equation}
    \label{e:taming}
    \lambda^m(t,\Psi_t^{m,*})
    = 
    \sigma_m^{(1)}(t)\exp\Big(-\max \{\Psi_t^{m,*}-\tilde\sigma,0\} / \tilde\sigma \Big)\,,
\end{equation}
where $\sigma_m^{(1)}$ is given by \eqref{eq:vola1}. 
In this case the dimension of the Brownian increment in \eqref{eq:euler1} is $d=2$. 

For the purposes of the simulation, $\Psi_t^{m,*}$ is given by \eqref{e:Psi_sim}, and the relevant parameters are \eqref{e:para} and the variance threshold is
\begin{equation}
    \label{e:V}
    \tilde \sigma = \Big( L^{10}_0 \Big)^2\,,
\end{equation}
which is the square of the initial $10$ year forward rate. 
This means that the threshold is assumed to correspond to a coefficient of variation (relative standard deviation) of the $10$ year yield of $100\%$.
It would also be possible to choose a different threshold corresponding to different maturities, however we find  that this only adds unnecessary complexity. 

As expected, this taming reduces the scenario variance, thereby making explosion very unlikely. This effect can be observed by looking at the histograms in Figure~\ref{fig:hist} of $L^1$ at times $10$, $20$, $30$ and $40$, or at the excess plots In Figures \ref{fig:excRMW} and \ref{fig:exc}.

The taming function is such that, once the threshold has been breached, the growth rate of the scenario variance approaches $0$.

Accordingly, there is a strong effect on cap prices. Indeed, once the scenario variance has increased beyond the threshold, the cap prices begin to decrease when compared to the mean-field independent case, compare Figures~\ref{fig:Caps_RMW} and \ref{fig:Caps}. Since we assume that the parameters, \eqref{e:paraRMW} or \eqref{e:para}, are obtained from a calibration routine based on the classical LMM, this poses restrictions on the applicability of the taming structure with respect to derivatives (such as caplets) whose values depend on  instantaneous volatilities. 

\begin{remark}
Due to Remark~\ref{rem:max}, the method of Section~\ref{sec:VS25} is covered by our existence and uniqueness Theorem~\ref{TH:TH1}. The variants in Sections~\ref{sec:VS6} and \ref{sec:VS4} are included because in the numerical study of their potential practical interest, but the corresponding existence and uniqueness problem is left for future research. 
\end{remark}

\subsection{Decorrelation beyond threshold}\label{sec:VS6}
In this section we assume a continuous decorrelation of rates as the observed variance $\Psi^{m,*}$ increases. 
Thus, the splitting \eqref{e:split} is realized as 
\begin{equation}
    \label{e:VS6}
    \lambda^m\Big(t,\Psi_t^{m,*}\Big)
    = 
    \left(
     \exp(-\Psi_t^{m,*}/ \tilde{\sigma})
      i_M  \sigma^{(1)}_m(t)
    +
    \Big(
    1 - \exp(-\Psi_t^{m,*}/\tilde{\sigma}) 
    \Big)e_m
    \right)/F
\end{equation}
where $e_m$ is the $m$-th standard vector in $\mathbb{R}^M$, $i_M: \mathbb{R}^2\to\mathbb{R}^M$ is the embedding along the first two factors, and $F = F(t,\Psi_t^{m,*})$ is normalization factor such that $|\lambda^m| = |\sigma^{(1)}_m|$. 
In this case the dimension of the Brownian increment in \eqref{eq:euler1} is $d=M$. 

If $\Psi_t^{m,*}$ is small compared to $\tilde{\sigma}$ then the system behaves approximately according to \eqref{eq:vola1}, and if $\Psi_t^{m,*}$ is large compared to $\tilde{\sigma}$ decorrelation sets in. 
The decorrelation approach is not quite as reliable as the taming function in reducing the probability of blow-up. However, compared to the standard LMM, the likeliness of explosion is still reduced significantly. See Figures~\ref{fig:hist_RMW} and \ref{fig:hist}  or the excess plots in Figures \ref{fig:excRMW} and \ref{fig:exc}.

Note that the instantaneous volatility, $\sigma_m^{(1)}$, remains unchanged in this setting. Therefore, caplet prices are unaffected.
Thus it can be expected, and is numerically verified in Figures~\ref{fig:Caps_RMW} and \ref{fig:Caps} that cap prices are preserved.
Concerning swaption prices, we observe a good level of replication under normal initial market conditions, represented by parameters~\eqref{e:paraRMW}, but significant deviations under excited conditions, represented by parameters~\eqref{e:para}.

\subsection{Anti-correlation beyond threshold}\label{sec:VS4}
The approach of Section~\ref{sec:VS6} is successful at mitigating explosion and preserving cap prices. However, as noted, there may be undesired effects on swaption prices.

Thus we consider the following anti-correlation prescription: 
\begin{equation} 
   \label{e:anti-cor}
    \lambda^m\Big(t,\Psi_t^{m,*}\Big)
    =
    \left\{
    \begin{matrix}
        \eqref{eq:vola1} && \quad\textup{if}\quad \Psi_t^{m,*} \le \tilde\sigma \\
        |\sigma_m^{(1)}(t)| e_n 
        &&\quad\textup{if}\quad \Psi_t^{m,*} > \tilde\sigma \textup{ and } m = 2n-1\\
        -|\sigma_m^{(1)}(t)| e_n 
        &&\quad\textup{if}\quad \Psi_t^{m,*} > \tilde\sigma \textup{ and } m = 2n
    \end{matrix}
    \right\}
\end{equation}
where $n=1,\dots,M/2$ (we assume that $M$ is even) and $e_n$ is the $n$-th standard basis vector.
In this case the dimension of the Brownian increment in \eqref{eq:euler1} is $d=M$.

\begin{figure}
    \centering
    \includegraphics[width=0.49\textwidth]{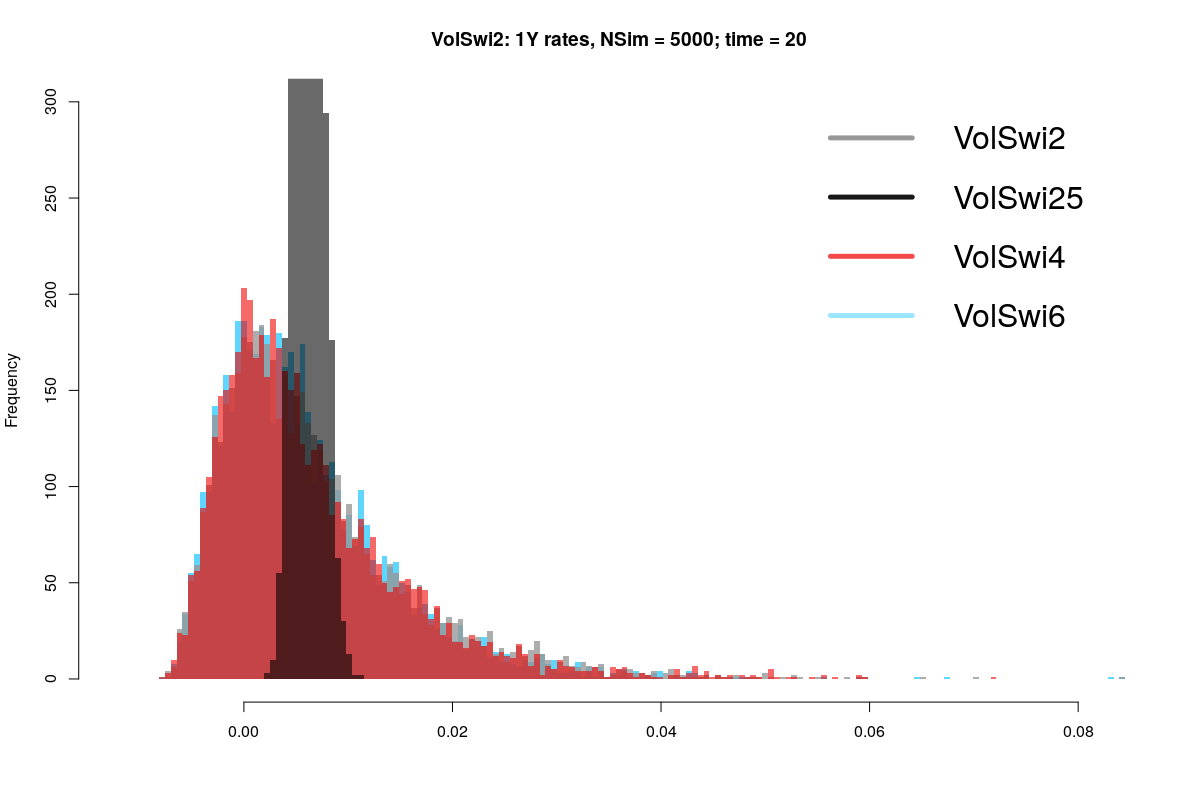}
    \includegraphics[width=0.49\textwidth]{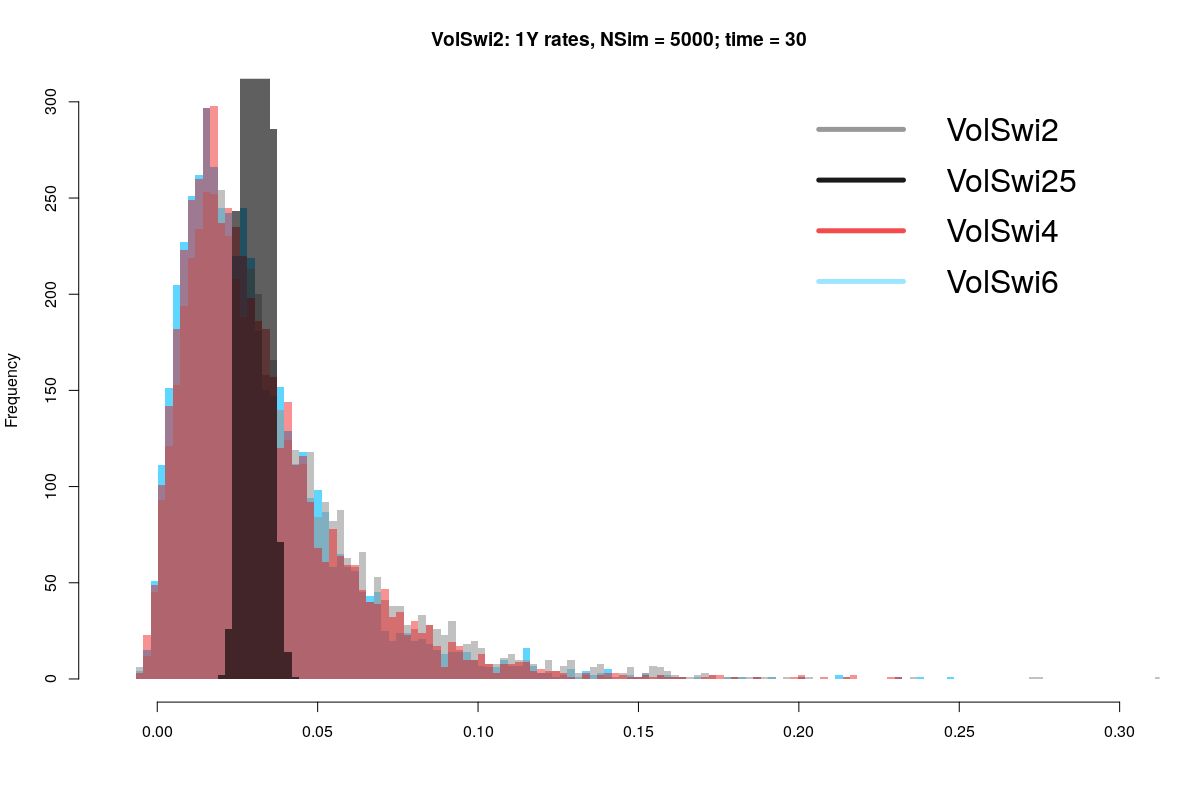}
    \includegraphics[width=0.49\textwidth]{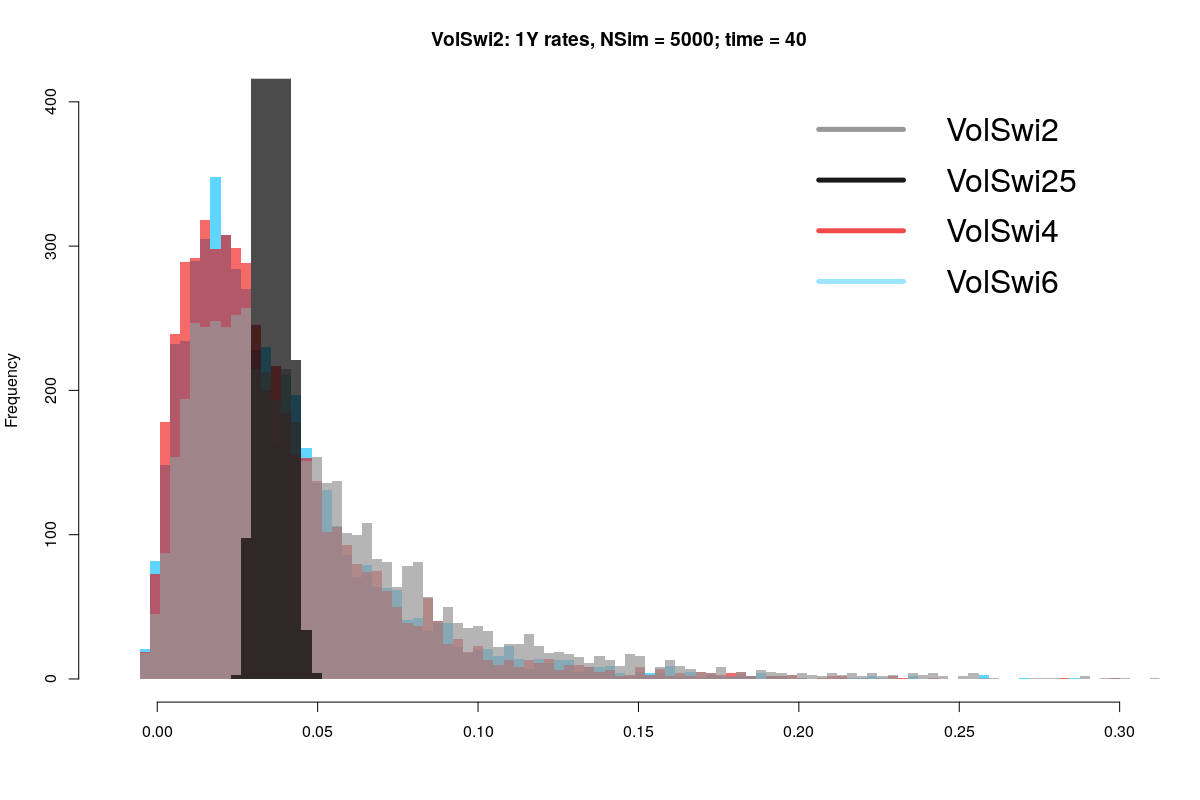}
    \includegraphics[width=0.49\textwidth]{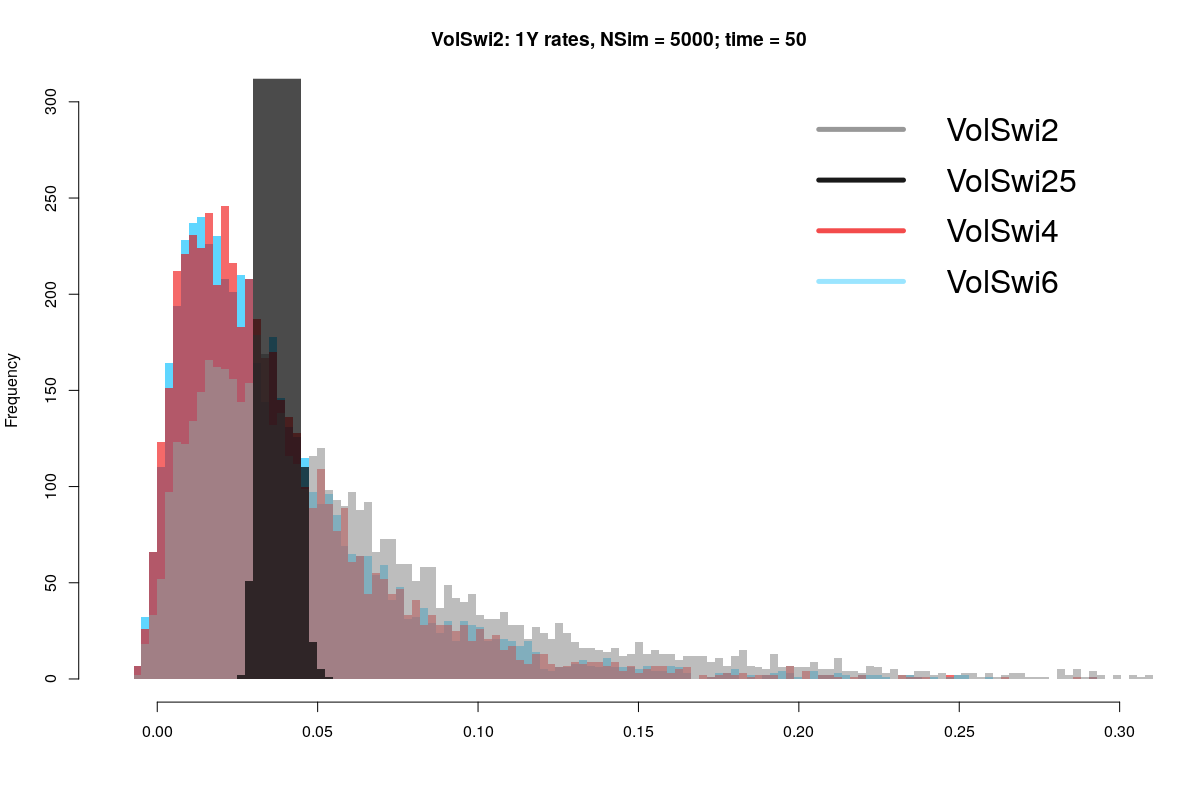}
    \caption{Histograms of $L^1$  generated according to \eqref{eq:vola1} at times $10$, $20$, $30$ and $40$. Parameters: \eqref{e:paraRMW}.}
    \label{fig:hist_RMW}
\end{figure}

\begin{figure}
    \centering
    \includegraphics[width=0.49\textwidth]{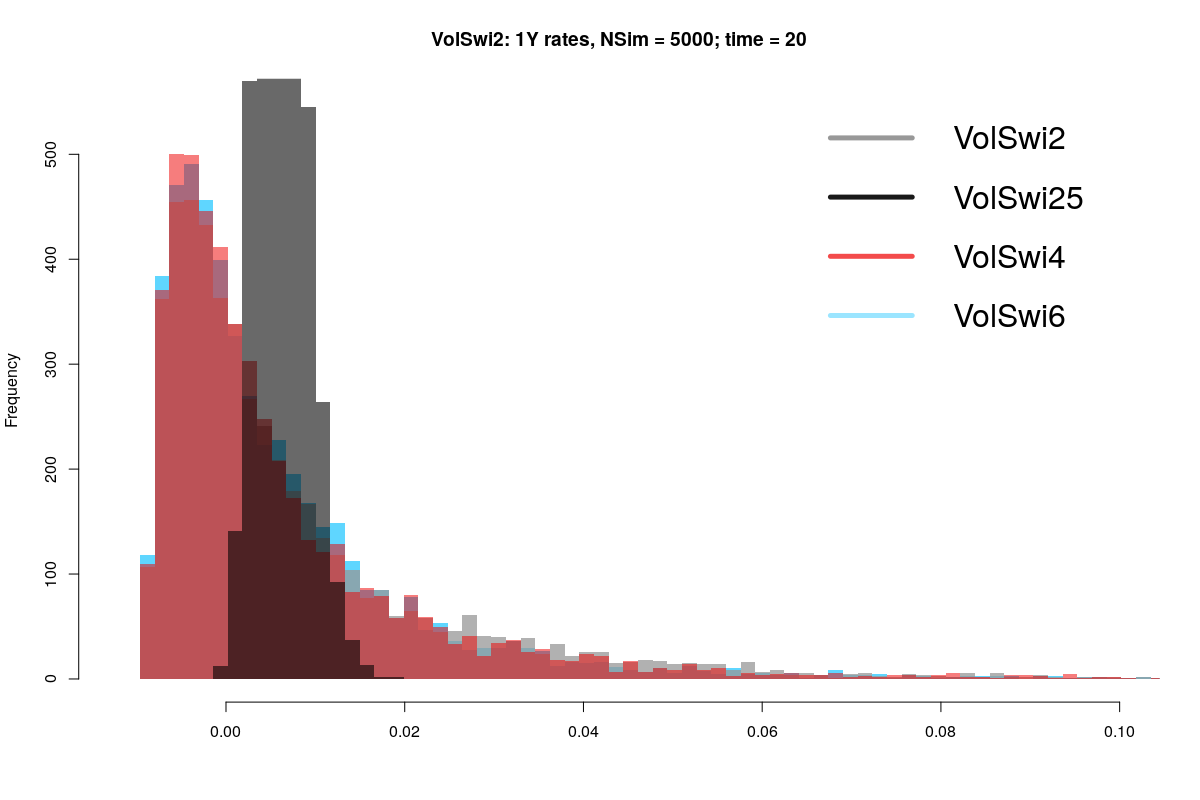}
    \includegraphics[width=0.49\textwidth]{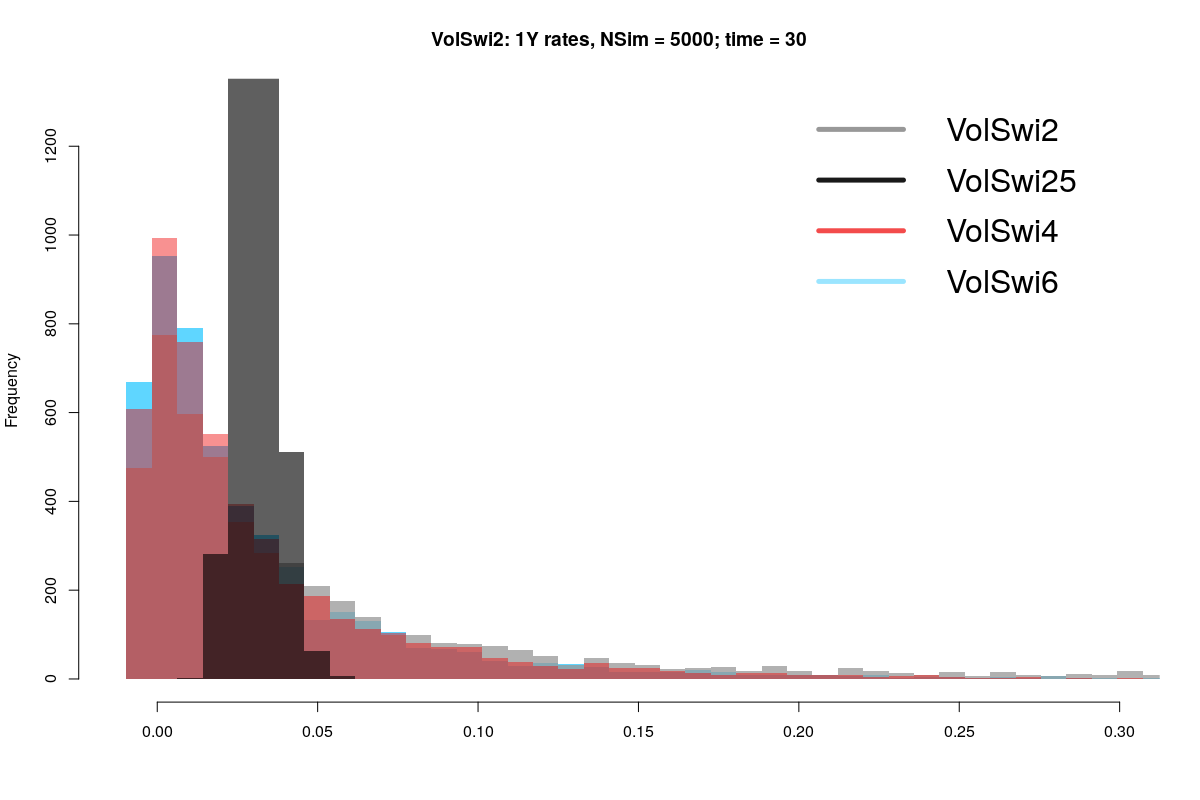}
    \includegraphics[width=0.49\textwidth]{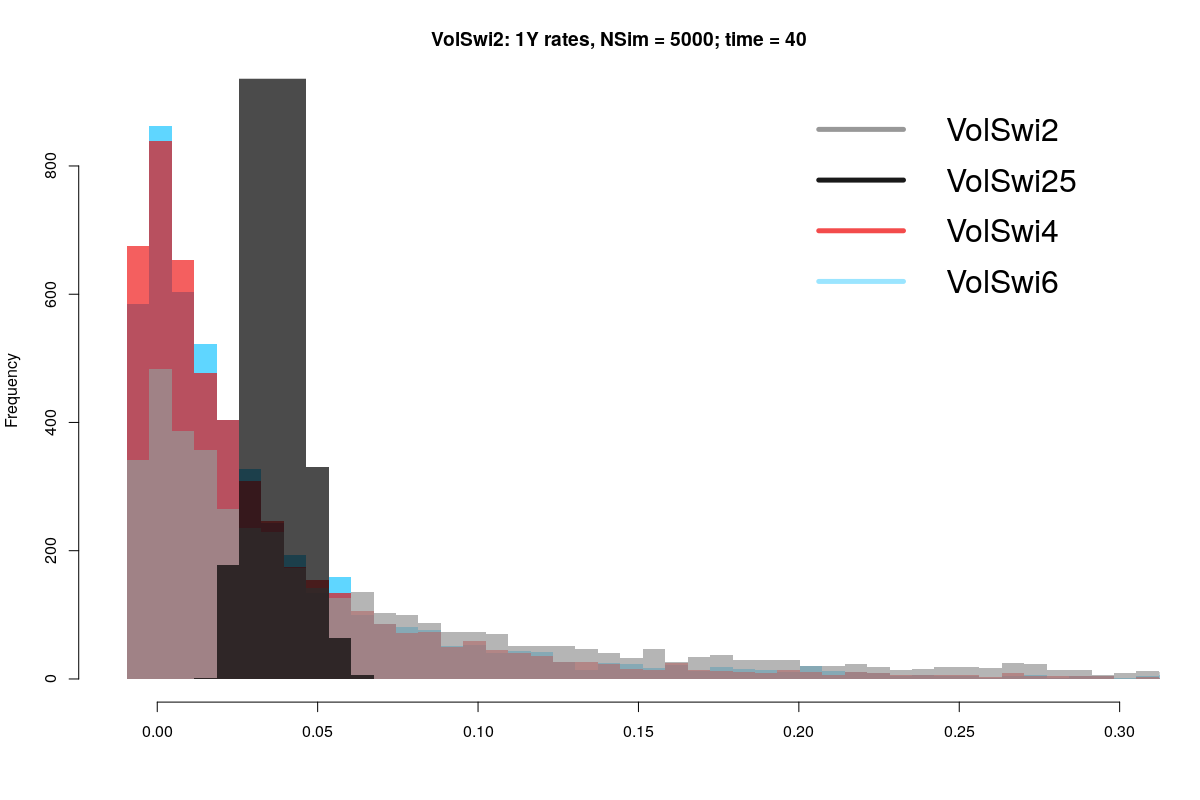}
    \includegraphics[width=0.49\textwidth]{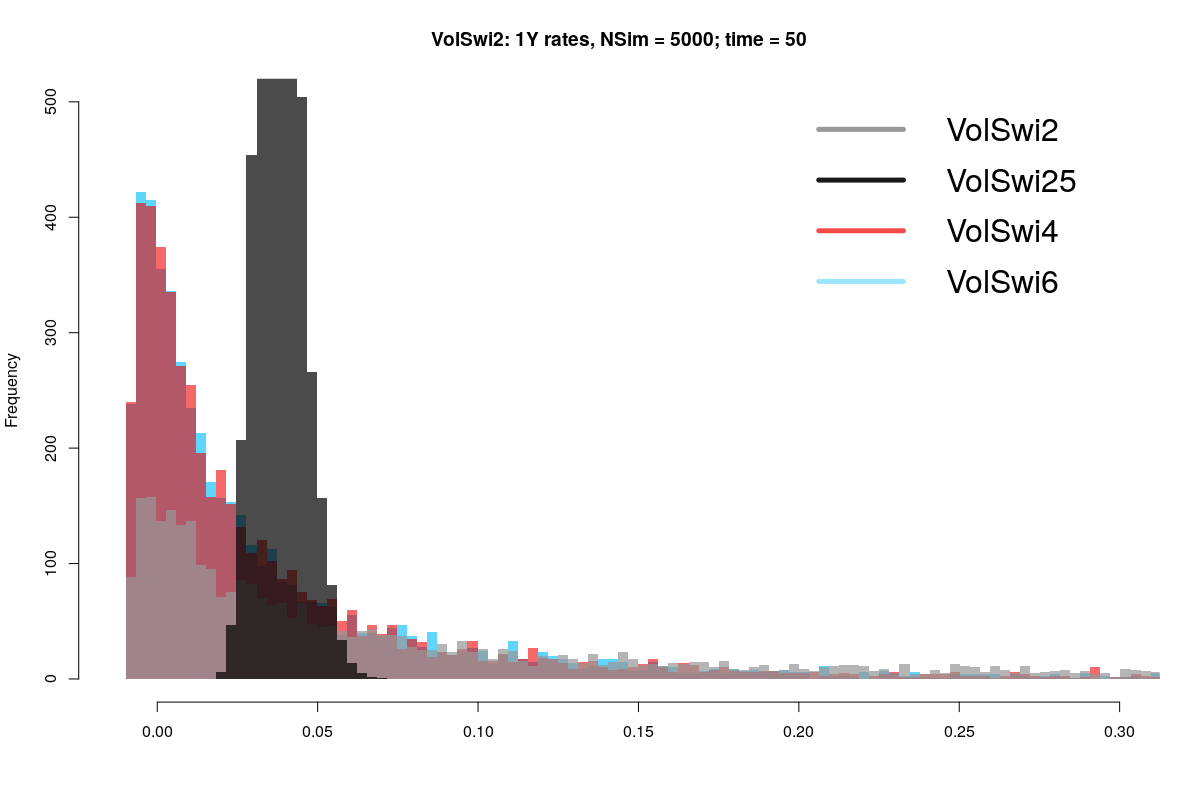}
    \caption{Histograms of $L^1$  generated according to \eqref{eq:vola1} at times $10$, $20$, $30$ and $40$. Parameters: \eqref{e:para}.}
    \label{fig:hist}
\end{figure}

Moreover, it is numerically verified that this choice (approximately) preserves caplet prices (Figure~\ref{fig:Caps_RMW} and \ref{fig:Caps}) and swaption prices (Figures~\ref{fig:SwaptionsRMW} and \ref{fig:Swaptions}), and significantly reduces blow-up (Figure~\ref{fig:hist}).
Finally, the time evolution of the percentage of scenarios exceeding $50\%$, resp.\ $100\%$, is shown in Figures~\ref{fig:excRMW} and \ref{fig:exc}.

\begin{figure}
    \centering
    \includegraphics[width=7cm]{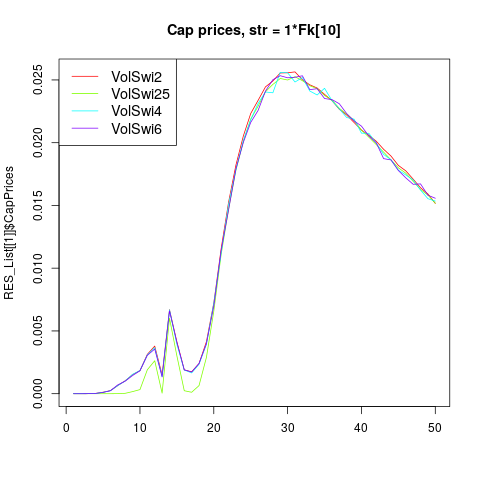}
    \includegraphics[width=7cm]{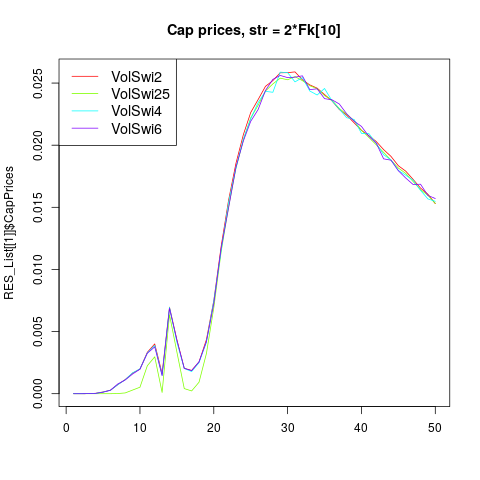}
    \includegraphics[width=7cm]{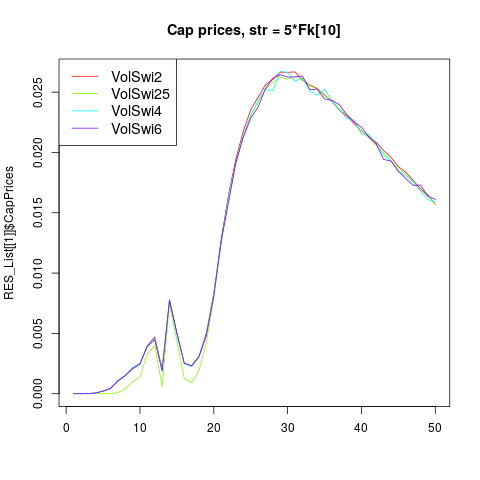}
    \includegraphics[width=7cm]{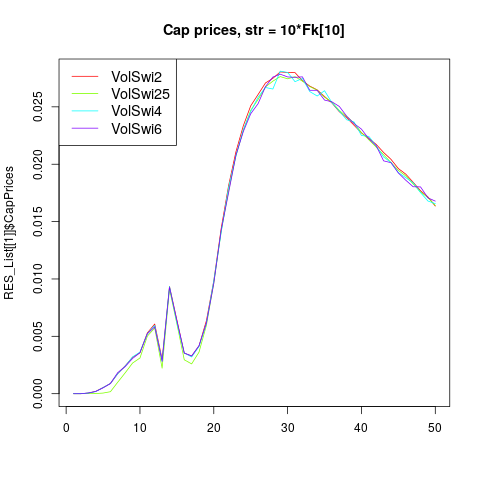}
    \caption{$1$-year caplet prices: VolSwi2 corresponds to Section~\ref{sec:VS2}, VolSwi25 to Section~\ref{sec:VS25}, VolSwi6 to Section~\ref{sec:VS6} and VolSwi4 to Section~\ref{sec:VS4}. The red line (VolSwi2) corresponds to the model calibration and is viewed as the `truth'. Note the deviation of the VolSwi25 line corresponding to the dampening assumption. Parameters: \eqref{e:paraRMW}}
    \label{fig:Caps_RMW}
\end{figure}

\begin{figure}
    \centering
    \includegraphics[width=7cm]{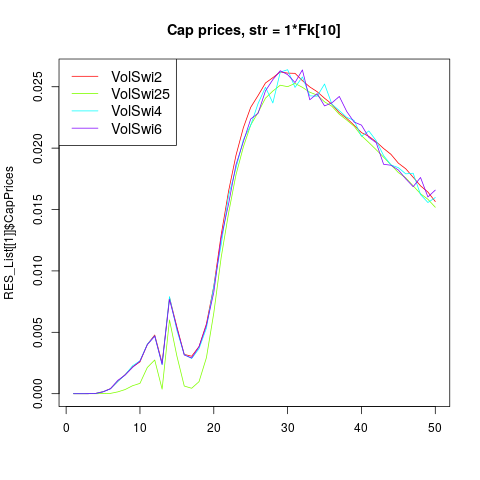}
    \includegraphics[width=7cm]{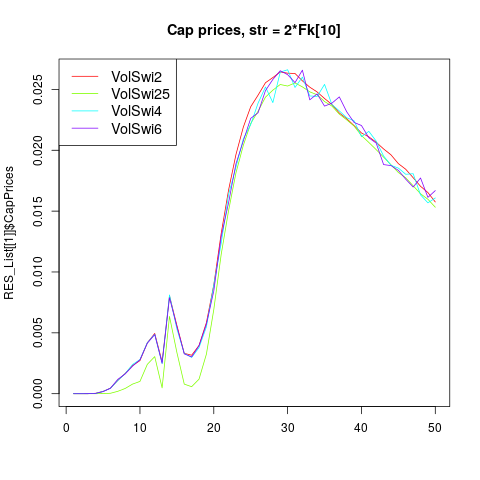}
    \includegraphics[width=7cm]{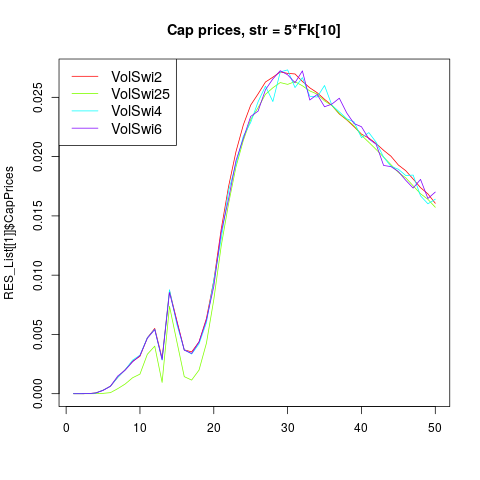}
    \includegraphics[width=7cm]{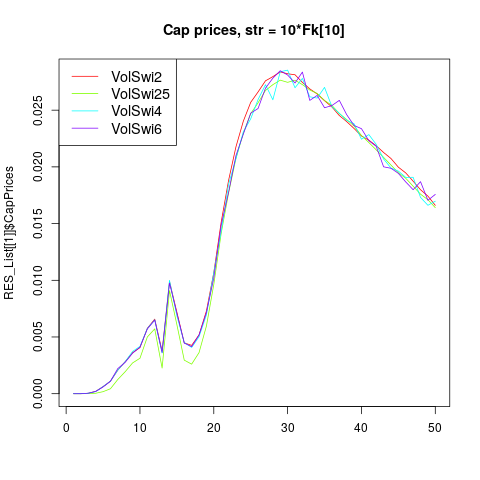}
    \caption{$1$-year caplet prices: VolSwi2 corresponds to Section~\ref{sec:VS2}, VolSwi25 to Section~\ref{sec:VS25}, VolSwi6 to Section~\ref{sec:VS6} and VolSwi4 to Section~\ref{sec:VS4}. The red line (VolSwi2) corresponds to the model calibration and is viewed as the `truth'. Note the deviation of the VolSwi25 line corresponding to the dampening assumption. Parameters: \eqref{e:para}}
    \label{fig:Caps}
\end{figure}

\begin{figure}
    \centering
    \includegraphics[width=7cm]{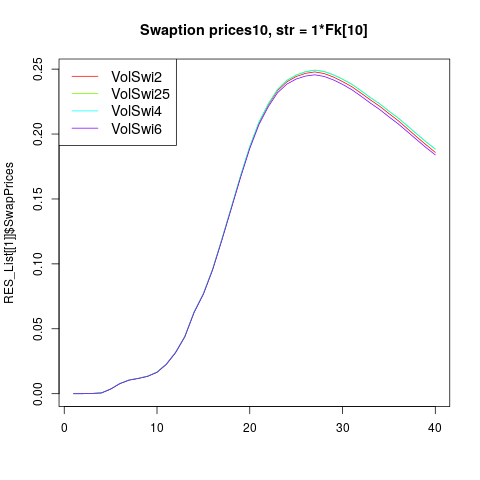}
    \includegraphics[width=7cm]{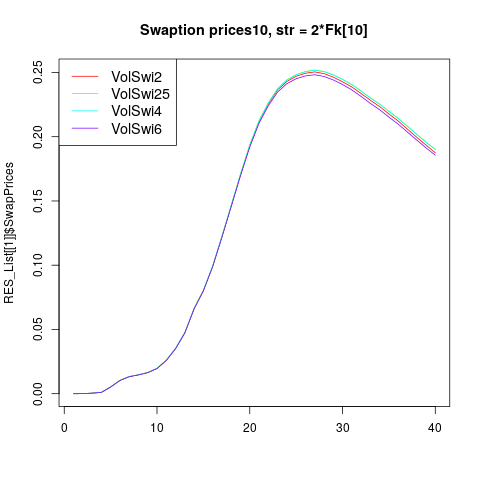}
    \includegraphics[width=7cm]{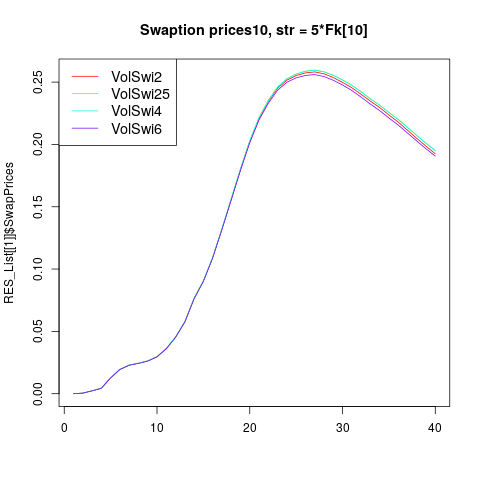}
    \includegraphics[width=7cm]{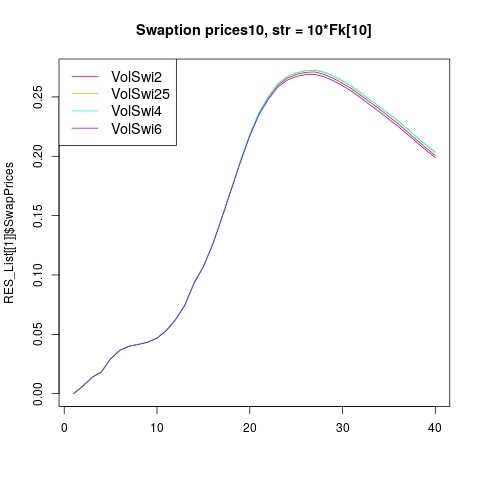}
    \caption{$10\times 10$ swaption prices: VolSwi2 corresponds to Section~\ref{sec:VS2}, VolSwi25 to Section~\ref{sec:VS25}, VolSwi6 to Section~\ref{sec:VS6} and VolSwi4 to Section~\ref{sec:VS4}. The red line (VolSwi2) corresponds to the model calibration and is viewed as the `truth'. Graphically, it is hardly distinguishable from VolSwi25 and VolSwi4. Parameters: \eqref{e:paraRMW}}
    \label{fig:SwaptionsRMW}
\end{figure}

\begin{figure}
    \centering
    \includegraphics[width=7cm]{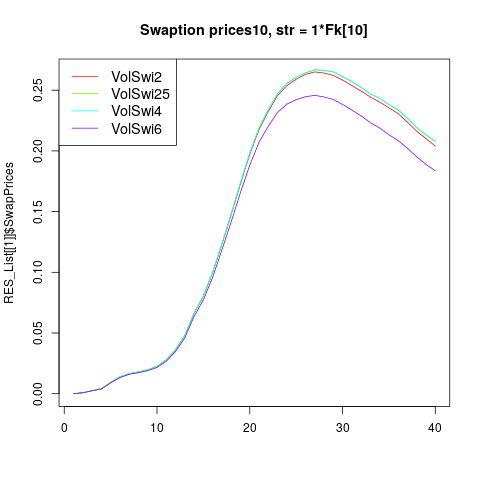}
    \includegraphics[width=7cm]{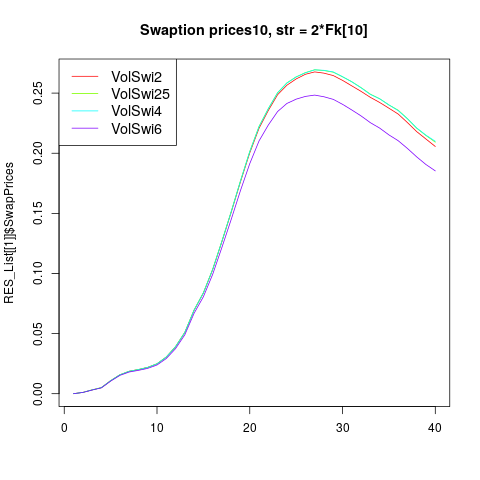}
    \includegraphics[width=7cm]{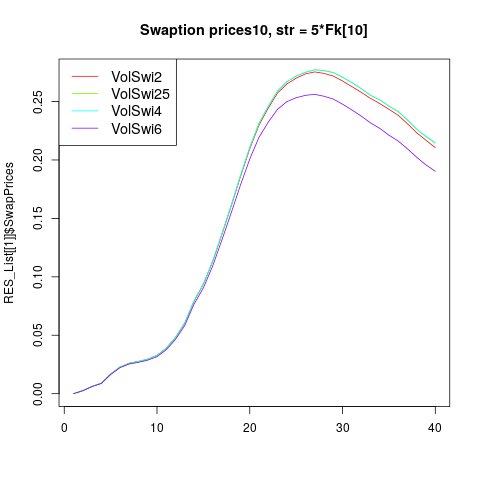}
    \includegraphics[width=7cm]{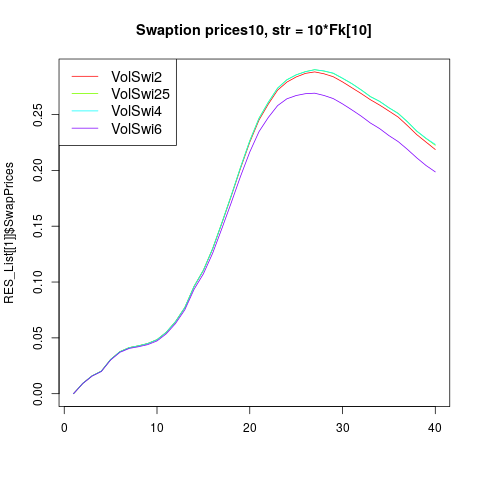}
    \caption{$10\times 10$ swaption prices: VolSwi2 corresponds to Section~\ref{sec:VS2}, VolSwi25 to Section~\ref{sec:VS25}, VolSwi6 to Section~\ref{sec:VS6} and VolSwi4 to Section~\ref{sec:VS4}. The red line (VolSwi2) corresponds to the model calibration and is viewed as the `truth'. Graphically, it is hardly distinguishable from VolSwi25 and VolSwi4. Parameters: \eqref{e:para}}
    \label{fig:Swaptions}
\end{figure}

\begin{figure}
    \centering
    \includegraphics[width=7cm]{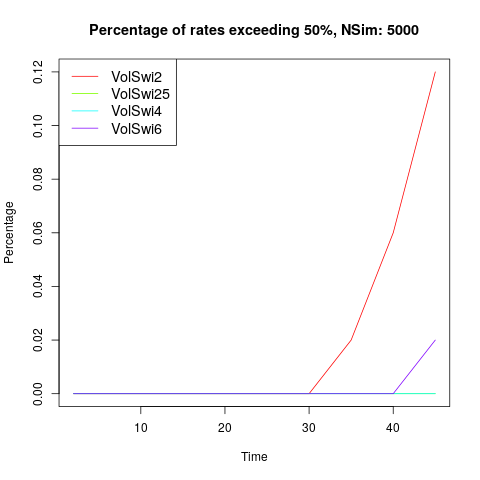}
    \includegraphics[width=7cm]{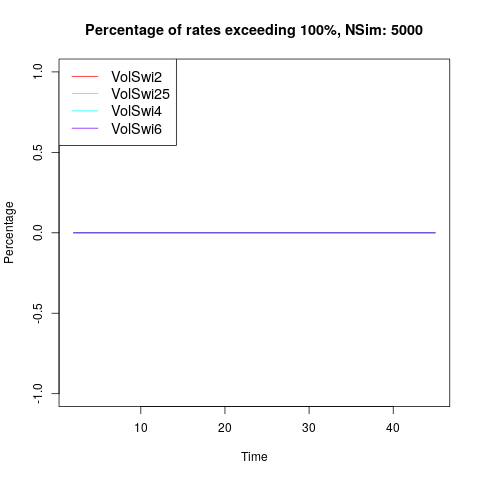}
    \caption{Excess plots: The time evolution of the percentage of scenarios exceeding $50\%$, resp.\ $100\%$, is shown. VolSwi2 corresponds to Section~\ref{sec:VS2}, VolSwi25 to Section~\ref{sec:VS25}, VolSwi6 to Section~\ref{sec:VS6} and VolSwi4 to Section~\ref{sec:VS4}.
    Parameters: \eqref{e:paraRMW}}
    \label{fig:excRMW}
\end{figure}

\begin{figure}
    \centering
    \includegraphics[width=7cm]{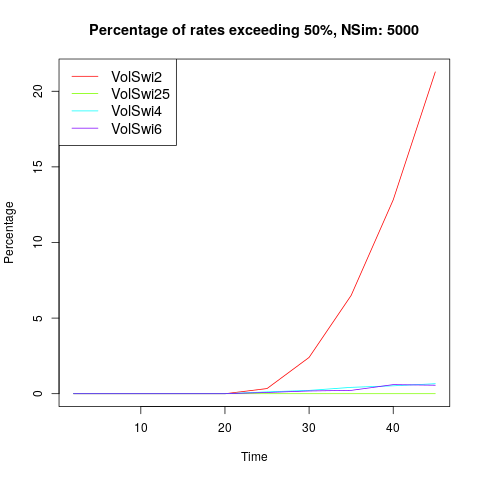}
    \includegraphics[width=7cm]{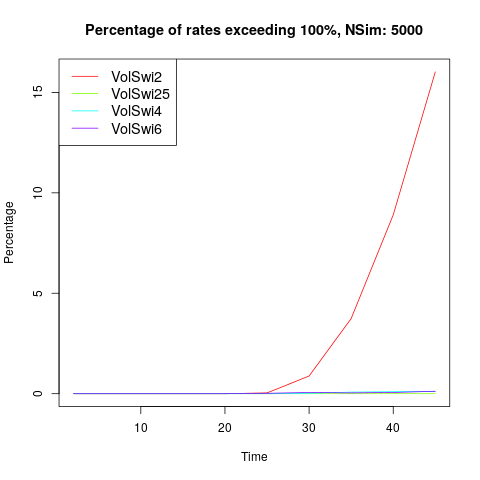}
    \caption{Excess plots: The time evolution of the percentage of scenarios exceeding $50\%$, resp.\ $100\%$, is shown. VolSwi2 corresponds to Section~\ref{sec:VS2}, VolSwi25 to Section~\ref{sec:VS25}, VolSwi6 to Section~\ref{sec:VS6} and VolSwi4 to Section~\ref{sec:VS4}.
    Parameters: \eqref{e:para}}
    \label{fig:exc}
\end{figure}

\newpage
\section{Existence and uniqueness of the solution to the underlying mean-field SDE}\label{Sec:WEll}


In what follows, we require the following notions and definitions:
\begin{itemize}
\item For a given $T >0$, we denote by $\mathscr{C} := C([0,T], \mathbb{R})$ the space of real-valued continuous functions endowed with the supremum norm,
\begin{equation*}
    \| f \|_t := \sup_{0 \leq s \leq t} |f_s|
\end{equation*} 
for $f \in \mathscr{C}$. The space $\mathscr{C}$ is also called path space. 
\item 
For $p \geq 2$, 
$\mathcal{S}^p([0,T])$ refers to the space of $\mathbb{R}^d$-valued progressively measurable, continuous processes, defined on the interval $[0,T]$, with bounded $p$-th moments, i.e., processes $(X_t)_{0 \leq t \leq T}$ satisfying $\mathbb{E} \left[ \|X \|_{T}^p \right] < \infty$. 
\item 
The set of probability measures on path-space $\mathscr{C}$ is denoted by
$\mathscr{P}(\mathscr{C})$ and the subset of 
square integrable
probability measures is
is denoted by
\begin{equation*}
 \mathscr{P}_2(\mathscr{C}) 
 =
    \Big\{
    \mu\in\mathscr{P}(\mathscr{C}):
    \mathbb{E}^{\mu}\left[ \| X \|^2_T \right] < \infty
    \Big\}.
\end{equation*}
\item 
As a metric on $\mathscr{P}_2(\mathscr{C})$, we use the following variant of the Wasserstein distance, see e.g., \cite{CD}: 
for $\mu, \nu \in \mathscr{P}_2(\mathscr{C})$ define
\begin{equation}
\label{e:W1}
    \mathbb{W}^{(2), \mathscr{C}}_T(\mu, \nu) 
    := 
    \left(\inf_{\pi \in \Pi(\mu,\nu)} 
    \int_{\mathscr{C} \times \mathscr{C}} \|x-y\|_T^2 \pi(\mathrm{d}x,\mathrm{d}y) \right)^{1/2},
\end{equation}
where $\Pi(\mu,\nu)$ denotes the set of couplings of $\mu$ and
$\nu$, i.e., $\pi \in \Pi(\mu,\nu)$ if and only if
$\Pi(\cdot,\mathscr{C})=\mu(\cdot)$ and $\Pi(\mathscr{C},\cdot)=\nu(\cdot)$.
We recall the definition of the standard $L_2$ Wasserstein distance: For any $\mu, \nu \in \mathscr{P}_2(\mathbb{R}^d)$, we define
\begin{equation}
\label{e:W2}
\mathbb{W}_{2}(\mu, \nu) := \left(\inf_{\pi \in \Pi(\mu,\nu)} \int_{\mathbb{R}^d \times \mathbb{R}^d} |x-y |^2 \pi(\mathrm{d}x,\mathrm{d}y) \right)^{1/2}.
\end{equation}
\end{itemize}

Since the coefficients of our MF-LMM are not Lipschitz continuous with respect to the $L_2$ Wasserstein distance as defined in \eqref{e:W1} and \eqref{e:W2}, common results on the existence and uniqueness of solutions to the corresponding mean-field SDEs do not apply. Hence, we need to show in full generality the existence and uniqueness of solutions for a class of mean-field SDEs involving such non-standard coefficients.  
Consider, for a given time horizon $[0,T]$, the following mean-field SDE on $\R^d$
\begin{equation}\label{EE1}
    \d X_t
    = b \left(t,X_t,\mu_t^{X} \right)\d t
       + \si(t,X_t,\mu_t^{X})\d W_t, 
    \quad X_0=\xi,
\end{equation}
where here $\mu_t^{X}$ is the marginal law of $X$ at the time $t \geq 0$, 
$b:[0,T] \times \Omega \times  \R^d \times \mathscr{P}_2(\R^d) \to \R^{d}$ and
$\si: [0,T] \times \Omega \times \R^d\times \mathscr
{P}_2(\R^d)\to \R^{d} \otimes \R^{m}$
are progressively measureable maps satisfying the assumptions stated below,
$\xi$ is an $\R^d$-valued random variable with bounded $p$-th moments (for a given $p \geq 2$), and $(W_t)_{t\ge0}$ is an
$m$-dimensional Brownian motion on the filtered probability space
$(\OO,\mathcal{F},(\mathcal{F}_t)_{t\ge0},\P)$. 

Here, we consider the case that $b$ and $\sigma$ are decomposable as
\begin{equation}
\label{e:dec}
 b(t,x,\mu) 
 = b_1(t,x,\mu) + xg(\mu)h_1(t,\omega), 
 \quad 
 \sigma(t,x,\mu) 
 = \sigma_1(t,x,\mu) + h_2(t,\omega)x^{T}g(\mu),
\end{equation}
where $g: \mathscr {P}_2(\R^d) \to \R$, $h_1:[0,T] \times \Omega \to \RR, h_2:[0,T] \times \Omega \to \RR^m$ and we will assume that, for any $x,y\in\R^d $, any $t \in [0,T]$, and $\mu,\nu\in\mathscr
{P}_2(\R^d)$:
\begin{enumerate}
\item[({\bf A}$_b^1$)] There exists a constant $L_b^1>0$ such that
\begin{align*}
|b_1(t,x,\mu)-b_1(t,y,\nu)| \le L_b^1(|x-y|+\mathbb{W}_2(\mu,\nu)).
\end{align*}
\item[({\bf A}$_\si^1$)] There exists a constant $L_\si^1>0$ such that
\begin{equation*}
\|\sigma_1(t,x,\mu)-\sigma_1(t,y,\nu)\| \le L_\si^1(|x-y|+\mathbb{W}_2(\mu,\nu)).
\end{equation*}
\item[({\bf A}$_{b\sigma}^1$)] The functions $h_1, h_2, g$ are uniformly bounded and there exists a constant $L_1>0$ such that 
\begin{align*}
|g(\mu) - g(\nu)| \leq L_1 \mathbb{E}\left[ |X+Y|^2 \right] \mathbb{W}_2(\mu,\nu),
\end{align*}
where $X$ and $Y$ are random variables with distribution, $\mu$ and $\nu$, respectively.
\item[({\bf A}$_{b\sigma}^2$)] There exists a constant $L_2>0$ (independent of $\mu$) such that 
\begin{align*}
|b_1(t,x,\mu)| + \|\sigma_1(t,x,\mu)\| \leq L_2(1 + |x|).
\end{align*}
\end{enumerate}

\begin{remark}
Note that the decompositions~\eqref{e:dec} are compatible with the form of the proposed volatility structures in Section~\eqref{sec:vol_str}.
\end{remark}

In what follows, we will prove that the mean-field SDE (\ref{EE1}), with deterministic initial data, indeed has a unique strong solution. We remark that generic constants $C>0$ might change their value from line to line in a chain of inequalities. 

\begin{thm}[Existence of a unique solution]\label{TH:TH2}
Let $X_0=x$, for some given value $x \in \RR^d$. 
Further, let assumptions ({\bf A}$_b^1$)--({\bf A}$_\si^1$) and ({\bf A}$_{b\sigma}^1$)--({\bf A}$_{b\sigma}^2$) be satisfied. 
Then, the mean-field SDE~\eqref{EE1},  
\begin{equation*}
    \d X_t
    = b \left(t,X_t,\mu_t^{X} \right)\d t
       + \si(t,X_t,\mu_t^{X})\d W_t, 
    \quad X_0=\xi,
\end{equation*}
has a unique strong solution in $\mathcal{S}^{p}([0,T])$, for any $p \geq 2$.
\end{thm}

\begin{proof}
For any given $\mu \in \mathscr{P}_2(\mathscr{C})$, we can reinterpret (\ref{EE1}) as a classical SDE
\begin{equation}\label{eq:Model1mu}
\mathrm{d} X_t^{\mu} = b^{\mu}(X_t^{\mu},t) \mathrm{d}t  + \sigma^{\mu}(X_t^{\mu},t) \mathrm{d}W_t, \quad X^{\mu}_0= x \in \mathbb{R}^d,
\end{equation}
where
\begin{align*}
b^{\mu}(X_t^{\mu},t) & :=  b(t,X_t^{\mu}, \mu_t) =  b_1(t,X_t^{\mu}, \mu_t) + X_t^{\mu}g(\mu_t)h_1(t) \\
\sigma^{\mu}(X_t^{\mu},t) & :=  \sigma(t,X_t^{\mu}, \mu_t) = \sigma_1(t,X_t^{\mu}, \mu_t) + h_2(t)X_t^{\mu,T}g(\mu_t)
\end{align*}
i.e., the coefficients do not depend on the law of $X_t^{\mu}$. Hence it can be seen as a classical (time-dependent) SDE, which has a unique strong solution in $\mathcal{S}^p([0,T])$, for $p \geq 2$ (see, e.g., \cite{XM}), i.e., there is a constant $C_p >0$ (independent of $\mu$, due to ({\bf A}$_{b\sigma}^1$)--({\bf A}$_{b\sigma}^2$)), such that
\begin{align*}
\mathbb{E}\left[\sup_{0 \leq t \leq T} |X_t^{\mu}|^p \right] \leq C_p.
\end{align*} 
In a next step, we introduce the map $\Phi: \mathscr{P}_2(\mathscr{C}) \to \mathscr{P}_2(\mathscr{C})$ by
\begin{equation*}
\Phi(\mu) = \text{Law}(X^{\mu}), 
\end{equation*}
i.e., for a fixed $\mu$, we solve the SDE (\ref{eq:Model1mu}) and set $\Phi(\mu)$ to be the law of the solution. 
Hence, $(X,\mu)$ is a solution of (\ref{EE1}) if and only if 
\begin{equation*}
X=X^{\mu} \text{ and }  \mu =  \Phi(\mu).
\end{equation*}
Consequently, we need to prove that $\Phi$ admits a unique fixed point, in order to show existence and uniqueness of a solution of (\ref{EE1}).  

Using the Lipschitz assumptions and applying the Burkholder-Davis-Gundy inequality yields, for a constant $C>0$ (changing its value from line to line), 

\begin{align*}
 \mathbb{W}^{(2),\mathscr{C}}_t(\Phi(\mu),\Phi(\nu))^2 &\leq \mathbb{E}\left[\sup_{0 \leq s \leq t} |X_s^{\mu} - X_s^{\nu} |^2 \right] \\
 & \leq C \mathbb{E} \left[\int_{0}^{t} |b_1(s,X_s^{\mu},\mu_s) - b_1(s,X_s^{\nu}, \nu_s)|^2 \mathrm{d}s \right]  \\
 & \quad + C \mathbb{E} \left[\int_{0}^{t} |X_s^{\mu}g(\mu_s)h_1(s) - X_s^{\nu} g(\nu_s)h_1(s) |^2 \mathrm{d}s \right]  \\
& \quad + C \mathbb{E} \left[\int_{0}^{t} \|h_2(s)X_s^{\mu,T}g(\mu_s) - h_2(s)X_s^{\mu,T}g(\nu_s) \|^2 \mathrm{d}s \right] \\ 
 & \quad + C \mathbb{E}  \left[ \sup_{0 \leq s \leq t} \left| \int_{0}^{s} ( \sigma_1(u,X_u^{\mu},\mu_u) - \sigma_1(u,X_u^{\nu},\mu_u)) \mathrm{d}W_u \right|^2 \right] \\
 & \leq C \mathbb{E} \left[ \int_{0}^{t} (\| X^{\mu} - X^{\nu} \|^2_s + \mathbb{W}_2(\mu_s,\nu_s)^2 )\mathrm{d}s \right] \\
 & \quad + C \mathbb{E} \left[\int_{0}^{t} |X_s^{\mu}g(\mu_s) - X_s^{\nu} g(\nu_s) |^2 \mathrm{d}s \right]. 
\end{align*}
Note that, employing the Lipschitz property and boundedness of $g$, ({\bf A}$_{b\sigma}^1$) and the fact that
\begin{equation*}
\mathbb{E}\left[\sup_{0 \leq t \leq T} |X_t^{\mu}|^p \right] < \infty,
\end{equation*}
where we recall that the bound is uniform in $\mu \in \mathscr{P}_2(\RR^d)$, due to ({\bf A}$_{b\sigma}^1$)--({\bf A}$_{b\sigma}^2$), we obtain the estimate 
\begin{align*}
 \mathbb{E} \left[\int_{0}^{t} |X_s^{\mu}g(\mu_s) - X_s^{\nu} g(\nu_s) |^2 \mathrm{d}s \right] 
 & \leq  
 C \mathbb{E} \left[ \int_{0}^{t} \| X^{\mu} - X^{\nu} \|^2_s \mathrm{d}s \right] + C \int_{0}^{t} \mathbb{E}\left[ |X_s^{\nu}|^2 \right] \mathbb{W}_2(\mu_s,\nu_s)^2 \mathrm{d}s \\
 & \leq  C_{\mu,\nu} \mathbb{E} \left[ \int_{0}^{t} (\| X^{\mu} - X^{\nu} \|^2_s + \mathbb{W}_2(\mu_s,\nu_s)^2 )\mathrm{d}s \right],
\end{align*}
where we used the fact that $\mathbb{E}[|X_s + Y_s|^2] \leq C_{\mu,\nu}$, with $X_s$ and $Y_s$ having distribution $\mu_s$ and $\nu_s$, respectively, and the boundedness of $g$. 

Consequently, Gronwall's inequality implies
\begin{equation*}
\mathbb{E}\left[\sup_{0 \leq s \leq t} |X_s^{\mu} - X_s^{\nu} |^2 \right]  \leq C_{\mu,\nu} \int_{0}^{t} \mathbb{W}_2(\mu_s,\nu_s)^2 \mathrm{d}s.
\end{equation*}
Hence, we arrive at
\begin{align}\label{eq:PIter}
 \mathbb{W}^{(2),\mathscr{C}}_t(\Phi(\mu),\Phi(\nu))^2 & \leq  \mathbb{E} \left[\sup_{0 \leq s \leq t} |X_s^{\mu} - X_s^{\nu} |^2 \right]  
 \leq C_{\mu,\nu} \int_{0}^{t} \mathbb{W}_2(\mu_s,\nu_s)^2 \mathrm{d}s 
  \leq  C_{\mu,\nu} \int_{0}^{t} \mathbb{W}^{(2),\mathscr{C}}_s(\mu,\nu)^2 \mathrm{d}s,
\end{align}
for any $t \in [0,T]$ with a constant $C_{\mu,\nu}>0$ depending on $\mu$ and $\nu$. 

Uniqueness follows from the previous inequality and another application of Gronwall's inequality. Existence is a consequence of a standard Picard-iteration type argument. We start with an arbitrary $\mu^0 \in \mathscr{P}_2(\mathscr{C})$ and then define the iterates $\mu^{k+1} = \Phi(\mu^k)$, for $k \geq 0$. This sequence forms a Cauchy sequence, which can be shown by iterating inequality (\ref{eq:PIter}) sufficiently often, with a limit that is a fixed point of $\Phi$.  
Therefore the SDE (\ref{EE1}) has unique strong solution in $\mathcal{S}^p([0,T])$, for $p \geq 2$. Also note that we have 
\begin{equation*}
\sup_{n \geq 0} \sup_{0 \leq t \leq T} \int_{\RR^d} |y|^p \mu^{X^{n}}_t(\mathrm{d}x) < \infty,
\end{equation*}
$p \geq 2$, which means that we have a uniform moment bound across all Picard-steps. Therefore, we conclude that the constant $C_{\mu,\nu}$ in \eqref{eq:PIter} only depends on the choice of $\mu^{0}$. 

\end{proof}
\begin{remark}[Link to the Existence of MFLMMs]\label{rem:max}
We note that if, both, drift and diffusion coefficient have the form 
\begin{equation*}
X_t g(t)e^{-\int_{\RR} x^2 \mu_t(\mathrm{d}x) + \left(\int_{\RR} x \mu_t(\mathrm{d}x) \right)^2},
\end{equation*}
where $\mu_t$ is the law of $X_t$ and $g: [0,T] \to \RR$ is a bounded function, then the assumptions of the above result are satisfied. To be precise, using the elementary inequality 
\begin{equation}\label{eq:usefull}
\left| \exp(x) - \exp(y) \right| \leq |x-y| \left(\exp(x) + \exp(y) \right)
\end{equation}
for all $x,y \in \RR$, we deduce, for any coupling $\pi$ of $\mu$ and $\nu$, that
\begin{align}
&\left|e^{-\int_{\RR} x^2 \mu(\mathrm{d}x) + \left(\int_{\RR} x \mu(\mathrm{d}x) \right)^2} - e^{-\int_{\RR} x^2 \nu(\mathrm{d}x) + \left(\int_{\RR} x \nu(\mathrm{d}x) \right)^2} \right| \nonumber \\
& \leq \left |-\int_{\RR} x^2 \mu(\mathrm{d}x) + \left(\int_{\RR} x \mu(\mathrm{d}x) \right)^2 + \int_{\RR} x^2 \nu(\mathrm{d}x) - \left(\int_{\RR} x \nu(\mathrm{d}x) \right)^2 \right| \times\nonumber \\
& \quad \times \left|e^{-\int_{\RR} x^2 \mu(\mathrm{d}x) + \left(\int_{\RR} x \mu(\mathrm{d}x) \right)^2} + e^{-\int_{\RR} x^2 \nu(\mathrm{d}x) + \left(\int_{\RR} x \nu(\mathrm{d}x) \right)^2} \right| \label{eq:usefullstart}\\
& \leq C \left| \int_{\RR^2} |x -y||x+y| \pi(\mathrm{d}x, \mathrm{d}y) \right| + \left| \int_{\RR^2} |x-y|  \pi(\mathrm{d}x, \mathrm{d}y)\right| \left| \int_{\RR^2} |x+y|  \pi(\mathrm{d}x, \mathrm{d}y)  \right|\nonumber \\
& \leq C \left( \int_{\RR^2} |x -y|^2 \pi(\mathrm{d}x, \mathrm{d}y) \right)^{1/2}  \left( \int_{\RR^2} |x + y|^2 \pi(\mathrm{d}x, \mathrm{d}y) \right)^{1/2} + \left| \int_{\RR^2} |x-y|  \pi(\mathrm{d}x, \mathrm{d}y)\right| \left| \int_{\RR^2} |x+y|  \pi(\mathrm{d}x, \mathrm{d}y)  \right|,\nonumber
\end{align}
for some constant $C>0$.
Since the above inequality holds for any coupling $\pi$, we also have 
\begin{align}\label{eq:WA}
&\left|e^{-\int_{\RR} x^2 \mu(\mathrm{d}x) + \left(\int_{\RR} x \mu(\mathrm{d}x) \right)^2} - e^{-\int_{\RR} x^2 \nu(\mathrm{d}x) + \left(\int_{\RR} x \nu(\mathrm{d}x) \right)^2} \right| \nonumber \\
& \leq C \mathbb{W}_2(\mu, \nu) \left( \left( \int_{\RR^2} |x + y|^2 \pi(\mathrm{d}x, \mathrm{d}y) \right)^{1/2} + \int_{\RR^2} |x+y|  \pi(\mathrm{d}x, \mathrm{d}y)  \right).
\end{align}
In particular, we also have in this case, for each Picard-iteration, a moment-bound of the SDE with fixed measure (\ref{eq:Model1mu}), which is independent of the present Picard-step. Based on \eqref{eq:usefull} one can derive
\begin{equation*}
\left|\exp\{-x^+\}-\exp\{-y^+\}\right|\leq |x-y|\left(\exp\{-x\}+\exp\{-y\}\right), 
\end{equation*}
for $x,\,y\in\RR$ and $x^+=\max\{x,0\}$. This inequality fortunately leads us to
\begin{align*}
 &\left| e^{-\left[\int_{\RR}x^2 \mu(\mathrm{dx})-\left(\int_Rx \mu(\mathrm{dx})\right)^2-\alpha\right]^+}-e^{-\left[\int_{\RR}x^2 \nu(\mathrm{dx})-\left(\int_Rx \nu(\mathrm{dx})\right)^2-\alpha\right]^+}\right|\\
 &\leq\left|\int_{\RR}x^2 \mu(\mathrm{dx})-\left(\int_Rx \mu(\mathrm{dx})\right)^2-\int_{\RR}x^2 \nu(\mathrm{dx})+\left(\int_Rx \nu(\mathrm{dx})\right)^2\right|\times\\
 &\quad\times\left(e^{-\int_{\RR}x^2 \mu(\mathrm{dx})+\left(\int_Rx \mu(\mathrm{dx})\right)^2+\alpha}+e^{-\int_{\RR}x^2 \nu(dx)+\left(\int_Rx \nu(\mathrm{dx})\right)^2+\alpha}\right),
\end{align*}
for some parameter $\alpha>0$. This resembles \eqref{eq:usefullstart} such that we can draw the same conclusions as in \eqref{eq:WA}. Terms of these form are relevant for a distribution dependent LIBOR market model, see (\ref{eq:diffusion}).
\end{remark}

\begin{remark}[Complementing the proof Theorem~\ref{TH:TH1}]\label{rem:ass}
In the proof of Theorem \ref{TH:TH1}, a coefficient depending on $\tilde{\mu}_t^i$, the law of 
\begin{equation*}
X_t:=(L_t^i,Z^{i,N}_t, \ldots, Z^{N-1,N}_t),
\end{equation*}
under $\mathcal{Q}^{N}$, appears. We analyse this coefficient for $i=N-1$ (similarly for other values of $i$). As shown in Section~\ref{sec:model}, this coefficient is
\begin{align*}
\tilde{\sigma}_{N-1}(t,\tilde{\mu}_t^{N-1}) = \sigma_{N-1}^{(1)}(t) e^{-\mathbb{E}^{\mathcal{Q}^N} \left[ \left( L_t^{N-1} \right)^2 Z^{N-1,N}_t   \right] + \left(\mathbb{E}^{\mathcal{Q}^N} \left[L_t^{N-1} Z^{N-1,N}_t   \right] \right)^2}.
\end{align*}
For two different measures $\tilde{\mu}^i,\tilde{\nu}^i \in \mathscr{P}_2(\mathscr{C})$, with marginals $\tilde{\mu}_t^i,\tilde{\nu}_t^i \in \mathscr{P}_2(\RR^2)$, we have  the estimate
\begin{align*}
& \left|e^{-\mathbb{E}^{\mathcal{Q}^N,\tilde{\mu}^i} \left[ \left( L_t^{N-1} \right)^2 Z^{N-1,N}_t   \right] + \left(\mathbb{E}^{\mathcal{Q}^N,\tilde{\mu}^i} \left[L_t^{N-1} Z^{N-1,N}_t   \right] \right)^2} -  e^{-\mathbb{E}^{\mathcal{Q}^N,\tilde{\nu}^i} \left[ \left( L_t^{N-1} \right)^2 Z^{N-1,N}_t   \right] + \left(\mathbb{E}^{\mathcal{Q}^N,\tilde{\nu}^i} \left[L_t^{N-1} Z^{N-1,N}_t   \right] \right)^2} \right| \\
& \leq C \mathbb{W}_2(\mu^{i}_t, \nu^{i}_t),
\end{align*} 
where $C>0$ depends on the moments (with respect to $\tilde{\mu}_t^i$, and $\tilde{\nu}_t^i$) of $L_t^{N-1}$ and $Z^{N-1,N}_t$. Also in this case the moments are uniformly bounded across all Picard-iterations and the proof of Theorem \ref{TH:TH1} also applies to this framework.
Furthermore, from the second part of Remark \ref{rem:max} we have that for the case of a measure dependence of the form,
$$
\exp \left\{
-\left[\mathbb{E}^{\mathcal{Q}^N,\tilde{\mu}^i} \left[ \left( L_t^{N-1} \right)^2 Z^{N-1,N}_t   \right] + \left(\mathbb{E}^{\mathcal{Q}^N,\tilde{\mu}^i} \left[L_t^{N-1} Z^{N-1,N}_t   \right] \right)^2-\alpha \right] ^+
\right\},
$$
the conclusion hold true as-well.
\end{remark}

\section{Conclusions}
We have introduced MF-LMM, a mean-field extension of the classical LIBOR Market Model. The main motivation for this is the  reduction of blow-up probability which is particularly relevant in the context of the valuation of long term guarantees. In this work we have studied the following aspects of MF-LMM:  

\begin{enumerate}
    \item[(1)] 
    Theorem~\ref{TH:TH1} proves existence and uniqueness of the MF-LMM, based on the results of Section~\ref{Sec:WEll}.
    \item[(2)] 
    Theorem~\ref{thm:cap} contains a Black formula  for a given measure flow in the mean-field setting.
    \item[(3)] 
    Section~\ref{sec:mfCal} adapts the Picard iteration construction of Theorem~\ref{TH:TH2} to devise a calibration algorithm based on Theorem~\ref{thm:cap}. The feasibility of this algorithm is shown in a numerical example.
    \item[(4)] In Section~\ref{sec:numerics}, we use an Euler-Maruyama discretization to simulate several variants of the MF-LMM. The numerical examples demonstrate (Figures~\ref{fig:hist_RMW} and \ref{fig:hist}) that a judicial choice of the mean-field dependence can lead to a reduction of blow-up probability (i.e., of number of scenarios breaching a certain predefined threshold), as expected. 
    \item[(5)] 
    The numerical approach of Section~\ref{sec:numerics} is tailor-made to facilitate the extension of classical LMMs to the mean-field realm. Thus, given the above mentioned well-posedness result, it can be readily applied to augment existing models -- whenever exploding rates are an issue.  
\end{enumerate}

Irrespective of blow-up considerations, one may also envisage to use the method of Theorem~\ref{thm:cap}, in conjunction with Section~\ref{sec:mfCal}, as additional degrees of freedom in the calibration process, e.g.\ if a fit to out-of-the-money caps is desirable. A possible topic for future studies is the extension of the Heath-Jarrow-Morton methodology to the mean-field setting.

\section*{Funding}

S. Desmettre is supported by the Austrian Science Fund (FWF) project  F5507-N26, which is part of the Special Research Program \textit{Quasi-Monte Carlo Methods: Theory and Applications}.\\
S. Thonhauser is supported by the Austrian Science Fund (FWF) project P33317.

\section*{Acknowledgements}

We wish to thank Wolfgang Stockinger for many fruitful and extensive discussions, in particular concerning the existence and uniqueness of the underlying mean-field SDE.

\end{document}